\documentclass[final,a4paper,11pt]{article}
\pdfoutput=1
\usepackage{graphicx}
\usepackage{mathptmx}
\usepackage{amssymb,amsmath, amsfonts,amsthm}
\usepackage{a4wide}
\usepackage{color}

\usepackage{hyperref}
\usepackage{float}
\usepackage{multirow}

\if{\usepackage{tikz}
\usepackage{pgfplots}
\pgfplotsset{compat=1.14}
\usepackage{comment}
\usepackage{cases}
\usepackage{array}
\usepackage{tabularx}
\usepackage{ragged2e}
\usepackage{comment}
\usepackage{subfigure}}\fi

\newcommand{\email}[1]{{\tt #1}}
\newcommand{\R}{\mathbb{R}}

\newcommand{\norm}[1]{\|#1\|}

\newcommand{\dist}[1]{{\rm dist}(#1)}

\newcommand{\mv}{\,\big |\,}
\newcommand{\bmv}{\,\Big |\,}
\newcommand{\bbmv}{\,\bigg |\,}

\newcommand{\A}{{\cal A}}
\newcommand{\B}{{\cal B}}

\newcommand{\I}{{\cal I}}

\newcommand{\M}{{\cal M}}
\newcommand{\New}{{\cal N}}

\newcommand{\Sp}{{\mathcal S}}
\newcommand{\F}{{\cal F}}
\newcommand{\Z}{{\cal Z}}

\newcommand{\longsetto}[1]{\mathop{\longrightarrow}\limits^{#1}}
\newcommand{\skalp}[1]{\langle #1\rangle}

\newcommand{\xb}{\bar x}
\newcommand{\yb}{\bar y}

\newcommand{\OO}{{\cal O}}

\newcommand{\gph}{\mathrm{gph}\,}
\newcommand{\dom}{\mathrm{dom}\,}

\newcommand{\tto}{\rightrightarrows}
\newcommand{\Limsup}{\mathop{{\rm Lim}\,{\rm sup}}}
\newcommand{\myvec}[1]{\begin{pmatrix}#1\end{pmatrix}}
\newcommand{\SCD}{SCD\ }

\newcommand{\reg}{{\rm reg\,}}
\newcommand{\ssstar}{semismooth$^{*}$ }
\newcommand{\ee}[2]{{#1}^{(#2)}}

\newcommand{\rge}{{\rm rge\;}}

\newlength{\myAlgBox}
\newtheorem{theorem}{Theorem}[section]
\newtheorem{proposition}[theorem]{Proposition}
\newtheorem{remark}[theorem]{Remark}
\newtheorem{lemma}[theorem]{Lemma}
\newtheorem{corollary}[theorem]{Corollary}
\newtheorem{definition}[theorem]{Definition}

\newtheorem{algorithm}{Algorithm}

\title{On the \SCD semismooth* Newton method \\ for generalized equations with application to a class of static contact problems with Coulomb friction}
\if{\title{On the numerical solution of a class of contact problems with Coulomb friction via the SCD \ssstar  Newton method }}\fi
\author{Helmut Gfrerer\thanks{Institute of Computational Mathematics, Johannes Kepler University
Linz, A-4040 Linz, Austria; \email{helmut.gfrerer@jku.at}}
\and Michael Mandlmayr\thanks{Institute of Computational Mathematics, Johannes Kepler University
Linz, A-4040 Linz, Austria; \email{michael.mandlmayr@live.com}}
 \and   Ji\v{r}\'{i} V. Outrata\thanks{Institute of Information Theory and Automation, Czech Academy of
 Sciences, 18208 Prague, Czech Republic, and Centre for
              Informatics and Applied Optimization, Federation University of Australia, POB 663,
              Ballarat,  Vic 3350, Australia,  \email{outrata@utia.cas.cz}}
\and{Jan Valdman\thanks{Institute of Information Theory and Automation, Czech Academy of
 Sciences, 18208 Prague, Czech Republic, and
 Department of Applied Mathematics, Faculty of Information Technology, Czech Technical University in Prague, Thákurova 9, 16000 Prague, Czech Republic, \email{jan.valdman@utia.cas.cz}}}}
\date{}

\begin{document}
\maketitle
{\bf Abstract.} In the paper, a variant of the \ssstar Newton method is developed for the numerical solution of generalized equations, in which the multi-valued part is a so-called SCD (subspace containing derivative) mapping. Under a rather mild regularity requirement, the method exhibits (locally) superlinear convergence behavior. From the main conceptual algorithm, two implementable variants are derived whose efficiency is tested via a generalized equation modeling a discretized static contact problem with Coulomb friction.

{\bf Keywords.} Newton method, semismoothness${}^*$, subspace containing derivative, generalized equation, Signorini problem with Coulomb friction.

{\bf AMS Subject classification.} 65K10, 65K15, 90C33.
\section{Introduction}
Consider the {\em generalized equation} (GE)
\begin{equation}\label{EqGE}
  0\in H(x):=f(x)+Q(x),
\end{equation}
where $f:\R^n\to\R^n$ is continuously differentiable and $Q:\R^n\tto\R^n$ is a set-valued mapping with a closed graph. For the numerical solution of (\ref{EqGE}) various methods are available, including the \ssstar Newton method developed in \cite{GfrOut21}. In this method, the approximation/linearization of the multi-valued term in (\ref{EqGE}) is performed on the basis of either the graph of the respective strict derivative or the limiting  coderivative. In each Newton step, one has to solve a linear square system with a non-singular matrix. The method thus differs both from the approach of Josephy (\cite{Jo79a}, \cite{Jo79b}), where the multi-valued part is not approximated at all, and from the Newton-type methods in \cite{Aze} and \cite{Dias}, where the multi-valued term is approximated in a different way.\\

In \cite{GfrOut22a} the authors have shown that for a class of the so-called SCD (subspace containing derivatives) mappings the \ssstar Newton can be improved. In particular, at these mappings, we dispose at each point with linear subspaces belonging to the graphs of the above-mentioned generalized derivatives, which generate the linear systems in the Newton step in a straightforward way. Moreover, the "regularity" requirement, needed to ensure the (locally) superlinear convergence, could have been substantially relaxed.
In \cite{GfrOutVal21} this so-called SCD \ssstar Newton method has been implemented in a class of {\em variational inequalities} (VIs) of the second kind that includes, among various problems of practical importance, also a class of discretized contact problems with {\em Tresca} friction (\cite{GfrOutValsozopol}). The very good performance of the new method, when applied to those problems, has led us to consider more complicated problems, in which $Q$
does not amount to the subdifferential of a convex function. This clearly requires a generalization of the theory of \cite{GfrOutVal21}. As a concrete representative problem of this type, we have chosen a discretized static contact problem, where the Tresca friction is replaced by the (physically more realistic) Coulomb friction.

Starting with the pioneering paper \cite{Necas} there are many papers and a comprehensive monograph \cite{EJ} devoted to static, quasi-static and dynamic contact problems with Coulomb friction for various types of material of the bodies in contact. Concerning the static contact of two elastic bodies or an elastic body with a rigid obstacle, it is known (\cite{Necas}) that this problem has a (not necessarily unique) solution whenever the friction coefficient belongs to the interval $(0,b]$, where $b > 0$ is a bound depending on the Poisson constant. This is a great difference from the corresponding discretized problems where, for small (mesh-dependent) values of the friction coefficient, one has to do with a {\em unique} solution. However, when the friction coefficient increases, discretized models may allow multiple solutions, as shown, e.g., in \cite{Renard}, \cite{Lig},

For simplicity, we consider only the contact of an elastic body with a rigid obstacle
(Signorini problem with friction). From the algebraic point of view, each two-body problem
can be rewritten formally as a one-body problem, see \cite{Wohl11}. Such problems can be modeled as quasi-variational inequalities (QVIs) expressed in terms of the so-called dual variables (having the physical meaning of stresses).
This enables us to employ a variety of methods developed for the numerical solution of QVIs, cf., e.g., \cite[Chapter 5]{GJKR}, \cite{FaKaSa13}. As shown in \cite{KO}, \cite{OKZ}, to these methods we can also count the classical semismooth Newton method ( \cite{Ku88}, \cite{QiSun93}). Another approach has been used in \cite{Bere1}, \cite{BeHaKoKuOut09}, where a solution is computed as a fixed point of a mapping generated by solving the corresponding contact problems with the Tresca friction. These problems can be solved, for example, by a specially tailored minimization routine; see \cite{Kuc}. In this paper, we use a new discrete model formulated in displacements, which is obtained by a modification of the GE used in \cite{BeHaKoKuOut09}. We thus work with a purely "primal" model, well suited for a direct application of the SCD \ssstar Newton method. \\

The plan of the paper is as follows: In Section 2 we recall some standard notions from variational analysis, which are extensively used throughout the paper. Section 3 is devoted to the SCD mappings; we list their basic properties and provide some calculus rules, indispensable in the construction of the
SCD \ssstar Newton method in Section 4 and its subsequent implementation. In Section 4 we present the main conceptual algorithm which is thereafter, in Section 5, implemented to the numerical solution of GE  (\ref{EqGE}). As a result, we obtain an efficient tool, applicable to a wide range of equilibrium models including VIs of the first and second kind, hemivariational inequalities (\cite{HaNi}) and many others. The first part of Section 6 is devoted to the construction of the new model of the discrete contact problem with Coulomb friction mentioned above (Section 6.1). In Section 6.2, it is shown that the multifunction, which arises in the respective GE, is an SCD mapping and possesses the \ssstar property. This finally enables us to specialize the formulas for the approximation and Newton step, developed in Section 5, for the GE considered. The results of the numerical experiments are presented in Section 7.\\

Our notation is basically standard. Given a linear subspace $L \subseteq \R^n$, $L^\perp$ denotes its orthogonal complement. For an element $u \in \R^n$, $\norm{u}$ denotes its
Euclidean norm,
$\B_\epsilon(u)$ denotes the closed ball around $u$ with radius $\epsilon$ and $\mathbb{S}_{\R^n}$ stands for the unit sphere in $\R^n$.  For a matrix $A$, $\rge A$ signifies its range. To avoid possible confusion, in some situations the dimension of a unit matrix $I$ will be indicated by a subscript ($I_n$) . Given a set $\Omega \subset \R^s$, we define the distance
from a point $x$ to $\Omega$ by
$d_{\Omega}(x) :=
\dist{x, \Omega} := \inf\{\norm{y-x}\mv y\in\Omega\}$, the respective indicator function is denoted by
$\delta_{\Omega}$ and $\stackrel{\Omega}{x \rightarrow \bar{x}}$ means convergence within $\Omega$.
When a mapping $F : \R^n \rightarrow \R^m$ is differentiable at $x$, we denote by $\nabla F(x)$ its Jacobian.

\section{Preliminaries}
Throughout the paper, we will frequently use the following basic notions of modern
variational analysis.

 \begin{definition}\label{DefVarGeom}
 Let $A$ be a set in $\mathbb{R}^{s}$, $\bar{x} \in A$ and $A$ be locally closed around $\bar{x}$. Then
\begin{enumerate}
 \item [(i)]The  {\em tangent (contingent, Bouligand) cone}
 to $A$ at $\bar{x}$ is given by
 \[
 %$
 T_{A}(\bar{x}):=\Limsup\limits_{t\downarrow 0} \frac{A-\bar{x}}{t}.
 %$
 \]
 %  and the {\em paratingent cone} to $A$ at $\xb$ is given by
 %\[
 %T^P_A(\xb):=\Limsup\limits_{\AT{x\setto{{A}}\xb}{t\downarrow 0}} \frac{A- x}{t}.
 %\]
 A tangent $u\in T_A(\xb)$ is called {\em derivable} if $\lim_{t\downarrow 0}\dist{\xb+tu,A}/t=0$. The set $A$ is {\em geometrically
derivable} at $\xb$ if every tangent vector $u$ to $A$ at $\xb$ is derivable.
 \item[(ii)] The set
 \[\widehat{N}_{A}(\bar{x}):=(T_{A}(\bar{x}))^{\circ}\]
 is the {\em regular (Fr\'{e}chet) normal cone} to $A$ at $\bar{x}$, and
 \[N_{A}(\bar{x}):=\Limsup\limits_{\stackrel{A}{x \rightarrow \bar{x}}} \widehat{N}_{A}(x)\]
 is the {\em limiting (Mordukhovich) normal cone} to $A$ at $\bar{x}$. \if{Given a direction $d
 \in\mathbb{R}^{s}$,
\[ N_{A}(\bar{x};d):= \Limsup\limits_{\stackrel{t\downarrow 0}{d^{\prime}\rightarrow
 d}}\widehat{N}_{A}(\bar{x}+ td^{\prime})\]
 is the {\em directional limiting normal cone} to $A$ at $\bar{x}$ {\em in direction} $d$.}\fi
 \end{enumerate}
\end{definition}
In this definition ''$\Limsup$'' stands for the Painlev\' e-Kuratowski {\em outer (upper) set limit}, see, e.g., \cite{RoWe98}.
%If $A$ is convex, then $\widehat{N}_{A}(\bar{x})= N_{A}(\bar{x})$ amounts to the classical normal cone in the sense of convex analysis and we will  write $N_{A}(\bar{x})$.
The above listed cones enable us to describe the local behavior of set-valued maps via various
generalized derivatives. Let $F:\R^n\tto\R^m$ be a (set-valued) mapping with the domain and the graph
\[\dom F:=\{x\in\R^n\mv F(x) \not = \emptyset\},\quad \gph F:=\{(x,y)\in\R^n\times\R^m \mv y\in F(x)\}.\]

\begin{definition}%\label{DefGenDeriv}
Consider a (set-valued) mapping $F:\R^n\tto\R^m$ and let $\gph F$ be locally closed around some $(\xb,\yb)\in \gph F$.
\begin{enumerate}
\item[(i)]
 The multifunction $DF(\xb,\yb):\R^n\tto\R^m$, given by $\gph DF(\xb,\yb)=T_{\gph F}(\xb,\yb)$, is called the {\em graphical derivative} of $F$ at $(\xb,\yb)$.
\item [(ii)]  The multifunction $D^\ast F(\xb,\yb ): \R^m \tto \R^n$,  defined by
 \[ \gph D^\ast F(\xb,\yb )=\{(y^*,x^*)\mv (x^*,-y^*)\in N_{\gph F}(\xb,\yb)\}\]
is called the {\em limiting  coderivative} of $F$ at $(\xb,\yb )$.
\end{enumerate}
\end{definition}

Let us now recall the following regularity notions.
\begin{definition}
  Let $F:\R^n\tto\R^m$ be a (set-valued) mapping and let $(\xb,\yb)\in\gph F$.
  \begin{enumerate}
    \item $F$ is said to be {\em metrically subregular at} $(\xb,\yb)$ if there exists a real $\kappa\geq 0$  along with some  neighborhood $X$ of $\xb$ such that
    \begin{equation*}%\label{EqSubreg}
      \dist{x,F^{-1}(\yb)}\leq \kappa\,\dist{\yb,F(x)}\ \forall x\in X.
    \end{equation*}
  %  The infimum over all $\kappa\geq 0$ such that \eqref{EqSubreg} holds $for some neighborhood $X$ is denoted by $\subreg F(\xb,\yb)$.
  \item $F$ is said to be {\em strongly metrically subregular at} $(\xb,\yb)$ if it is metrically subregular at $(\xb,\yb)$ and there exists a neighborhood $X'$ of $\xb$ such that $F^{-1}(\yb)\cap X'=\{\xb\}$.
  \item $F$ is said to be {\em metrically regular around} $(\xb,\yb)$ if there is some $\kappa\geq 0$ together with neighborhoods $X$ of $\xb$ and $Y$ of $\yb$ such that
      \begin{equation*}%\label{EqMetrReg}
      \dist{x,F^{-1}(y)}\leq \kappa\,\dist{y,F(x)}\ \forall (x,y)\in X\times Y.
    \end{equation*}
%    The infimum over all $\kappa\geq 0$ such that \eqref{EqMetrReg} holds %for some neighborhoods $X,Y$ is denoted by $\reg F(\xb,\yb)$.
  \item $F$ is said to be {\em strongly metrically regular around} $(\xb,\yb)$ if it is metrically regular around $(\xb,\yb)$ and $F^{-1}$ has a single-valued localization around $(\yb,\yb)$, i.e., there are  open neighborhoods $Y'$ of $\yb$, $X'$ of $\xb$ and a mapping $h:Y'\to\R^n$ with $h(\yb)=\xb$ such that $\gph F\cap (X'\times Y')=\{(h(y),y)\mv y\in Y'\}$.
  \end{enumerate}
\end{definition}

It is easy to see that the strong metric regularity around $(\xb,\yb)$ implies the strong metric subregularity  at $(\xb,\yb)$ and the metric regularity around $(\xb,\yb)$ implies the metric subregularity at $(\xb,\yb)$. To check the metric regularity one often employs the so-called
Mordukhovich criterion, see, e.g. \cite[Theorem 9.43]{RoWe98}, according to which this property around $(\xb,\yb)$ is equivalent to the condition
\begin{equation}\label{EqMoCrit}
 0\in D^*F(\xb,\yb)(y^*)\ \Rightarrow\ y^*=0.
\end{equation}
Another useful characterization of metric regularity in terms of the graphical derivative is provided by the so-called Aubin-criterion by Dontchev et al \cite{DonQuiZl06}.
For pointwise characterizations of the other stability properties from Definition 2.3, the reader is referred to \cite[Theorem 2.7]{GfrOut22a}.

In this preparatory section, we end with a definition of the \ssstar property, which paved the way for both the \ssstar Newton method in \cite{GfrOut21} and the SCD \ssstar Newton method in \cite{GfrOut22a}.
\begin{definition}\label{Defssstar}
We say that $F:\R^n\tto\R^n$ is {\em \ssstar}  at $(\xb,\yb)\in\gph F$ if
for every $\epsilon>0$ there is some $\delta>0$ such that the inequality
\begin{align*}%\label{EqDefSemiSmooth}
\vert \skalp{x^*,x-\xb}-\skalp{y^*,y-\yb}\vert&\leq \epsilon
\norm{(x,y)-(\xb,\yb)}\norm{(x^*,y^*)}
\end{align*}
holds for all $(x,y)\in \gph F\cap \B_\delta(\xb,\yb)$ and all $(y^*,x^*)$ belonging to $\gph D^\ast F(x,y)$.
\end{definition}

\section{On SCD mappings}
\subsection{Basic properties}
In this section, we wish to recall the basic definitions and characteristics of the \SCD property  introduced in the recent paper \cite{GfrOut22a}.

In what follows, we denote by $\Z_n$ the metric space of all $n$-dimensional subspaces of $\R^{2n}$ equipped with the metric
\[d_{\Z_n}(L_1,L_2):=\norm{P_{L_1}-P_{L_2}},\]
where $P_{L_i}$ is the symmetric $2n\times 2n$ matrix representing the orthogonal projection onto $L_i$, $i=1,2$.

Sometimes, we will also work with bases for the subspaces $L\in\Z_n$. Let $\M_n$ denote the collection of all $2n\times n$ matrices with full column rank $n$ and define for $L\in \Z_n$ the set
\[\M(L):=\{Z\in \M_n\mv \rge Z =L\},\]
i.e., the columns of $Z\in\M(L)$ create a basis for $L$.

We treat every element of $\R^{2n}$ as a column vector. To keep the notation simple, we write $(u,v)$ instead of $\myvec{u\\v}\in\R^{2n}$ when this does not cause confusion.

Let $L\in\Z_n$ and consider $Z\in \M(L)$. Then we can divide $Z$ into two $n\times n$ matrices $A$ and $B$ and write $Z=(A,B)$ instead of $Z=\myvec{A\\B}$. It follows that $\rge(A,B):=\{(Au,Bu)\mv u\in\R^n\}\doteq \{\myvec{Au\\Bu}\mv u\in\R^n\}=L$.

Furthermore, for every $L\in \Z_n$ we can define the {\em adjoint} space
\begin{align}\label{EqDualSubspace}
  L^*&:=\{(-v^*,u^*)\mv (u^*,v^*)\in L^\perp\}.
\end{align}
It can be shown that $(L^*)^*=L$ and $d_{\Z_n}(L_1,L_2)=d_{\Z_n}(L_1^*,L_2^*)$. Thus, the mapping $L\to L^*$ defines an isometry on $\Z_n$.

\begin{definition}\label{DefSCDProperty}
  Consider a mapping $F:\R^n\tto\R^n$.
  \begin{enumerate}
    \item We call $F$ {\em graphically smooth of dimension $n$} at $(x,y)\in \gph F$, if $T_{\gph F}(x,y)=\gph DF(x,y)\in \Z_n$. In addition, we denote by $\OO_F$ the set of all points where $F$ is graphically smooth of dimension $n$.
    \item We associate with $F$ the four mappings $\widehat\Sp F:\gph F\tto \Z_n$, $\widehat\Sp^* F:\gph F\tto \Z_n$, $\Sp F:\gph F\tto \Z_n$, $\Sp^* F:\gph F\tto \Z_n$, given by
    \begin{align*}\widehat\Sp F(x,y)&:=\begin{cases}\{\gph DF(x,y)\}& \mbox{if $(x,y)\in\OO_F$,}\\
    \emptyset&\mbox{else,}\end{cases}\\
    \widehat\Sp^* F(x,y)&:=\begin{cases}\{\gph DF(x,y)^*\}& \mbox{if $(x,y)\in\OO_F$,}\\
    \emptyset&\mbox{else,}\end{cases}\\
    \Sp F(x,y)&:=\Limsup_{(u,v)\longsetto{{\gph F}}(x,y)} \widehat\Sp F(u,v) \\
    &=\{L\in \Z_n\mv \exists (x_k,y_k)\longsetto{{\OO_F}}(x,y):\ \lim_{k\to\infty} d_{\Z_n}(L,\gph DF(x_k,y_k))=0\},\\
    \Sp^* F(x,y)&:=\Limsup_{(u,v)\longsetto{{\gph F}}(x,y)} \widehat\Sp^* F(u,v)\\
    &=\{L\in \Z_n\mv \exists (x_k,y_k)\longsetto{{\OO_F}}(x,y):\ \lim_{k\to\infty} d_{\Z_n}(L,\gph DF(x_k,y_k)^*)=0\}.
    \end{align*}
    \item We say that $F$ has the {\em\SCD (subspace containing derivative) property at} $(x,y)\in\gph F$, if $\Sp^*F(x,y)\not=\emptyset$. $F$ is said to have the \SCD property {\em around} $(x,y)\in\gph F$, if there is a neighborhood $W$ of $(x,y)$ such that $F$ has the \SCD property at every $(x',y')\in\gph F\cap W$.
     Finally, we call $F$ an {\em \SCD mapping} if
    $F$ has the \SCD property at every point of its graph.
  \end{enumerate}
\end{definition}
Since $L\to L^*$ is an isometry on $\Z_n$ and $(L^*)^*=L$, the mappings $\Sp^*F$ and $\Sp F$ are related through
  \begin{equation}\label{EqSp*Sp}\Sp^* F(x,y)=\{L^*\mv L\in \Sp F(x,y)\},\ \Sp F(x,y)=\{L^*\mv L\in \Sp^* F(x,y)\}.\end{equation}
The name \SCD property is motivated by the following statement.
\begin{lemma}[cf.{\cite[Lemma 3.7]{GfrOut22a}}] Let $F:\R^n\tto\R^n$ and let $(x,y)\in \gph F$. Then $L\subseteq \gph D^*F(x,y)$ $\forall L\in \Sp^*F(x,y)$.
\end{lemma}
In the recent paper \cite{GfrOut22a} one can find several calculus rules to work with the \SCD property including the next result.
\begin{proposition}[{cf.\cite[Proposition 3.15]{GfrOut22a}}]\label{PropSum}
  Let $F:\R^n\tto\R^n$ have the \SCD property at $(x,y)\in\gph F$ and let $h:U\to\R^n$ be continuously differentiable at $x\in U$ where $U\subseteq\R^n$ is open. Then $F+h$ has the \SCD property at $(x,y+h(x))$ and
  \begin{align}\label{EqSCDSum1}\Sp (F+h)(x,y+h(x))=\Big\{\left(\begin{matrix}I&0\\\nabla h(x)&I\end{matrix}\right)L\mv L\in \Sp F(x,y)\Big\},\\
  \label{EqSCDSum2}\Sp^* (F+h)(x,y+h(x))=\Big\{\left(\begin{matrix}I&0\\\nabla h(x)^T&I\end{matrix}\right)L\mv L\in \Sp^* F(x,y)\Big\}.
  \end{align}
\end{proposition}
Note that these sum rules are also valid at the points $(x,y)\in\gph F$ where $F$ does not have the \SCD property: In this case, we simply have $\Sp (F+h)(x,y+h(x))=\Sp^*(F+h)(x,y+h(x))=\Sp F(x,y)=\Sp^* F(x,y)=\emptyset$. In addition, we will need some calculus rules for the Cartesian product of mappings. Consider the mapping $F:\prod_{i=1}^p\R^{n_i}\tto \prod_{i=1}^p\R^{n_i}$ defined by
  \begin{equation}\label{EqCartProd}F(x_1,\ldots,x_p):=\prod_{i=1}^pF_i(x_i),\end{equation}
  where each multifunction $F_i:\R^{n_i}\tto\R^{n_i}$, $i=1,\ldots,p$, has a closed graph. Note that
  \begin{align}\label{EqProdEqual}\lefteqn{T_{\gph F}\big((x_1,\ldots,x_p),(y_1,\ldots,y_p)\big)}\\
  \nonumber
  &=\left\{\big((u_1,\ldots,u_p),(v_1,\ldots,v_p)\big)\mv \big((u_1,v_1),\ldots,(u_p,v_p)\big)\in T_{\gph F_1\times\ldots\times\gph F_p}\big((x_1,y_1),\ldots,(x_p,y_p)\big)\right\}\\
  \label{EqInclProd}&\subset\left\{\big((u_1,\ldots,u_p),(v_1,\ldots,v_p)\big)\mv (u_i,v_i)\in T_{\gph F_i}(x_i,y_i),\ i=1,\ldots,p\right\},\end{align}
  where the first equation follows from the identity
  \[\gph F=\left\{\big((x_1,\ldots,x_p),(y_1,\ldots,y_p)\big)\mv \big((x_1,y_1),\ldots,(x_p,y_p)\big)\in \gph F_1\times\ldots\times\gph F_p\right\}\]
  together with \cite[Exercise 6.7]{RoWe98} and the inclusion \eqref{EqInclProd} is a consequence of \cite[Proposition 6.41]{RoWe98}.
\begin{lemma}\label{LemSCDProduct}
  Let $F$ be given by \eqref{EqCartProd} and let $(x,y):=\big((x_1,\ldots,x_p),(y_1,\ldots,y_p)\big)\in\gph F$. Then we have the following:
  \begin{enumerate}
    \item If $(x,y)\in\OO_F$ then each tangent cone $T_{\gph F_i}(x_i,y_i)$, $i=1,\ldots,p$,  is a subspace of $\R^{2n_i}$.
    \item On the contrary, if $(x_i,y_i)\in\OO_{F_i}$, $i=1,\ldots,p$, and all but at most one of the sets $\gph F_i$, $i=1,\ldots,p$,  are geometrically derivable at $(x_i,y_i)$, then $(x,y)\in\OO_F$.
  \end{enumerate}
  \end{lemma}
\begin{proof}In order to prove the first assertion, let $(x,y)\in\OO_F$, let $i\in\{1,\ldots,p\}$ be arbitrarily fixed and consider the set
\[L_i:=\{0_{m_1}\}\times T_{\gph F_i}(x_i,y_i)\times\{0_{m_2}\}\quad\mbox{with}\quad m_1:=2\sum_{k=1}^{i-1}n_k,\ m_2:=2\sum_{k=i+1}^{p}n_k.\]
Then it readily follows from Definition \ref{DefVarGeom}(i) that $L_i$ is a subset of $T:=T_{\gph F_1\times\ldots\times\gph F_p}\big((x_1,y_1),\ldots,(x_p,y_p)\big)$, which is a subspace by \eqref{EqProdEqual}. Consider two tangents $t_1,t_2\in T_{\gph F_i}(x_i,y_i)$ together with two scalars $\mu_1,\mu_2$. Since $(0_{m_1},t_j,0_{m_2})\in L_i\subset T$, $j=1,2$, we conclude
\[\mu_1(0_{m_1},t_1,0_{m_2})+\mu_2(0_{m_1},t_2,0_{m_2})=(0_{m_1},\mu_1t_1+\mu_2t_2t_1,0_{m_2})\in T\]
and from \eqref{EqInclProd} we deduce $\mu_1t_1+\mu_2t_2\in T_{\gph F_i}(x_i,y_i)$. This proves that $T_{\gph F_i}(x_i,y_i)$ is a subspace.

The second statement follows from the fact that under the stated assumption, inclusion \eqref{EqInclProd} holds with equality, cf. \cite[Proposition 1]{GfrYe17a}.
\end{proof}
\begin{proposition}\label{PropSCDProduct}
  Let $F$ be given by \eqref{EqCartProd} and assume that all the mappings $F_i$, $i=1,\ldots,p$, are \SCD mappings. If all, but at most one, of the mappings $F_i$, $i=1,\ldots,p$, have the property that $\gph  F_i$ is geometrically derivable at every point $(x_i,y_i)\in\OO_{F_i}$, then $F$ is an \SCD mapping, and
  \begin{equation}
    \label{EqInclOO_F}\OO_F\supset\left\{\big((x_1,\ldots,x_p),(y_1,\ldots,y_p)\big)\mv (x_i,y_i)\in\OO_{F_i},\ i=1,\ldots,p\right\}.
  \end{equation}
  Moreover, for every $(x,y)=\big((x_1,\ldots,x_p),(y_1,\ldots,y_p)\big)\in\gph F$ there holds
  \begin{gather}
    \label{EqInclSpProd}\Sp F(x,y)\supset\Big\{\left\{\big((u_1,\ldots,u_p),(v_1,\ldots,v_p)\big)\mv (u_i,v_i)\in L_i\right\}\mv L_i\in \Sp F_i(x_i,y_i),\ i=1,\ldots,p\Big\},\\
    \label{EqInclSp*Prod}\Sp^* F(x,y)\supset\Big\{\left\{\big((v_1^*,\ldots,v_p^*),(u_1^*,\ldots,u_p^*)\big)\mv (v_i^*,u_i^*)\in L_i\right\}\mv L_i\in \Sp^* F_i(x_i,y_i),\ i=1,\ldots,p\Big\}.
  \end{gather}
  The equality holds in the inclusions \eqref{EqInclOO_F}, \eqref{EqInclSpProd} and \eqref{EqInclSp*Prod} if, in addition, all but at most one of the mappings $F_i$, $i=1,\ldots,p$, have the following property: For every $(x_i,y_i)\in\gph F_i$ such that $T_{\gph F_i}(x_i,y_i)$ is a subspace, the dimension of this subspace is $n_i$.
\end{proposition}
\begin{proof}The inclusions \eqref{EqInclOO_F} and \eqref{EqInclSpProd} follow immediately from Lemma \ref{LemSCDProduct}. Since $F_i$, $i=1,\ldots, p$, are \SCD mappings, $\OO_{F_i}$ is dense in $\gph F_i$ and from \eqref{EqInclOO_F} we conclude that $\OO_F$ is dense in $\gph F$. This proves that $F$ is an \SCD mapping. Now consider subspaces $L_i\in \Sp F_i(x_i,y_i)$, $i=1,\ldots,p$. By taking into account the relation
\[\left\{\big((u_1,\ldots,u_p),(v_1,\ldots,v_p)\big)\mv (u_i,v_i)\in L_i\right\}^*=\left\{\big((v_1^*,\ldots,v_p^*),(u_1^*,\ldots,u_p^*)\big)\mv (v_i^*,u_i^*)\in L_i^*\right\},\]
the inclusion \eqref{EqInclSp*Prod} follows from \eqref{EqInclSpProd} and \eqref{EqSp*Sp}. The statement about equality in \eqref{EqInclOO_F} is again a consequence of Lemma \ref{LemSCDProduct} implying equality in \eqref{EqInclSpProd} and \eqref{EqInclSp*Prod}.
\end{proof}

The following large class of graphically Lipschitzian mappings covers many mappings important in applications; cf. \cite{Ro85}, and is also important in the context of \SCD mappings.
\begin{definition}[cf.{\cite[Definition 9.66]{RoWe98}}]\label{DefGraphLip}A mapping $F:\R^n\tto\R^m$ is {\em graphically Lipschitzian of dimension $d$} at $(\xb,\yb)\in\gph F$ if there is an open neighborhood $W$ of $(\xb,\yb)$ and a one-to-one mapping $\Phi$ from $W$ onto an open subset of $\R^{n+m}$ with $\Phi$ and $\Phi^{-1}$ continuously differentiable, such that $\Phi(\gph F\cap W)$  is the graph of a Lipschitz continuous mapping $f:U\to\R^{n+m-d}$, where $U$ is an open set in $\R^d$.
\end{definition}

Every multifunction $F:\R^n\tto\R^n$ that is graphically Lischitzian of dimension $n$ at some point $(x,y)\in\gph F$, has the \SCD property around $(x,y)$ by \cite[Proposition 3.17]{GfrOut22a}. We now state another property related to Proposition \ref{PropSCDProduct}.

\begin{lemma}\label{LemGraphLisch}Let $F:\R^n\tto\R^n$ be graphically Lipschitzian of dimension $n$ at $(\xb,\yb)\in\gph F$. Then there is an open neighborhood $W$ of $(\xb,\yb)$ such that for all $(x,y)\in\gph F\cap W$ the following properties hold:
\begin{enumerate}
  \item[(i)]If $(x,y)\in\OO_F$ then $\gph F$ is geometrically derivable at $(x,y)$.
  \item[(ii)]If $T_{\gph F}(x,y)$ is a subspace, then the dimension of this subspace is $n$.
\end{enumerate}
\end{lemma}
\begin{proof}
  Regarding property (i), we refer to \cite[Remark 3.18]{GfrOut22a} and \cite[Proposition 8.41]{RoWe98}. To show the second statement, let $W$, $\Phi$, $U$ and $f$ be as in Definition \ref{DefGraphLip} and consider $(x,y)\in \gph F\cap W$ such that $T_{\gph F}(x,y)$ is a subspace. Denoting $(u,f(u))= \Phi(x,y)$, we obtain that
  \[T_{\gph f}(u,f(u))=\nabla \Phi(x,y)T_{\gph F}(x,y)\]
  is a subspace with the same dimension as $T_{\gph F}(x,y)$. Therefore, the graphical derivative $Df(u,f(u))$ is a linear mapping and  $f$ is Fr\'echet differentiable at $u$ by \cite[Exercise 9.25]{RoWe98}. Since $T_{\gph f}(u,f(u))=\rge(I,\nabla f(u))$, the dimension of $T_{\gph f}(u,f(u))$ is $n$, which is the same as the dimension of $T_{\gph F}(x,y)$.
\end{proof}

Next, we turn to the notion of \SCD regularity.
\begin{definition}
\begin{enumerate}
\item We denote by $\Z_n^{\rm reg}$ the collection of all subspaces $L\in\Z_n$ such that
  \begin{equation*}%\label{EqSCDReg_L}
    (y^*,0)\in L\ \Rightarrow\ y^*=0.
  \end{equation*}
  \item  A mapping $F:\R^n\tto\R^n$ is called {\em \SCD regular around} $(x,y)\in\gph F$, if $F$ has the \SCD property around $(x,y)$ and
  \begin{equation}\label{EqSCDReg}
    (y^*,0)\in L \Rightarrow\ y^*=0\ \forall L\in \Sp^*F(x,y),
  \end{equation}
  i.e., $L\in \Z_n^{\rm reg}$ for all $L\in \Sp^*F(x,y)$. Further, we will denote by
  \[{\rm scd\,reg\;}F(x,y):=\sup\{\norm{y^*}\mv (y^*,x^*)\in L, L\in \Sp^*F(x,y), \norm{x^*}\leq 1\}\] %\sup\{\norm{C_L}\mv L\in \Sp^*F(x,y)\}\]
  the {\em modulus of \SCD regularity} of $F$ around $(x,y)$.
\end{enumerate}
\end{definition}
Since the elements of $\Sp^*F(x,y)$ are contained in $\gph D^*F(x,y)$, it follows from the Mordukhovich criterion \eqref{EqMoCrit} that \SCD regularity is weaker than metric regularity, and consequently \SCD regularity is also weaker than strong metric regularity.

In the following proposition, we state some basic properties of subspaces $L\in\Z_n^{\rm reg}$.
\begin{proposition}[cf.{\cite[Proposition 4.2]{GfrOut22a}}] \label{PropC_L}
    We have $L\in \Z_n^{\rm reg}$ if and only if for every $(A,B)\in\M(L)$ the matrix $B$ is not singular. Thus,
    for every $L\in \Z_n^{\rm reg}$   there is a unique $n\times n$ matrix $C_L$  such that $L=\rge(C_L,I)$. Further, $L^*=\rge(C_L^T,I)\in\Z_n^{\rm reg}$,
    \begin{equation*}%\label{EqC_L}
    \skalp{x^*,C_L^Tv}=\skalp{y^*,v}\ \forall (y^*,x^*)\in L\forall v\in\R^n.
    \end{equation*}
    and
    \begin{equation*}%\label{EqKappa_L}
    \norm{y^*}\leq \norm{C_L}\norm{x^*}\ \forall (y^*,x^*)\in L.
  \end{equation*}
\end{proposition}
Note that for every $L\in\Z_n^{\reg}$ there is $C_L=AB^{-1}$ for all $(A,B)\in \M(L)$.
Combining \cite[Equation  (34), Lemma 4.7 and Proposition 4.8]{GfrOut22a} we obtain the following lemma.
\begin{lemma}\label{LemSCDReg}
   Assume that $F:\R^n\tto\R^n$ is \SCD regular around $(\xb,\yb)\in\gph F$. Then
   \[{\rm scd\,reg\;}F(\xb,\yb)=\sup\{\norm{C_L}\mv L\in\Sp^*F(\xb,\yb)\}<\infty.\]
   Moreover, $F$ is \SCD regular around every $(x,y)\in\gph F$ sufficiently close to $(\xb,\yb)$ and
  \[\limsup_{(x,y)\longsetto{\gph F}(\xb,\yb)}{\rm scd\,reg\;}F(x,y)\leq{\rm scd\,reg\;}F(\xb,\yb).\]
\end{lemma}

\section{On semismooth* Newton methods for SCD mappings}

In this section we recall the general framework for the \ssstar Newton method introduced in \cite{GfrOut21} and adapted to \SCD mappings in \cite{GfrOut22a}.
Consider the inclusion
\begin{equation}\label{EqIncl}
  0\in F(x),
\end{equation}
where $F:\R^n\tto\R^n$ is a mapping having the \SCD property around some point $(\xb,0)\in\gph F$.

The following notion relaxes the \ssstar property from Definition \ref{Defssstar}.
\begin{definition}%\label{DefSCDssstar}
We say that $F:\R^n\tto\R^n$ is {\em\SCD \ssstar}  at $(\xb,\yb)\in\gph F$ if $F$ has the \SCD property around $(\xb,\yb)$ and
for every $\epsilon>0$ there is some $\delta>0$ such that the inequality
\begin{align*}%\label{EqDefSCDSemiSmooth}
\vert \skalp{x^*,x-\xb}-\skalp{y^*,y-\yb}\vert&\leq \epsilon
\norm{(x,y)-(\xb,\yb)}\norm{(x^*,y^*)}
\end{align*}
holds for all $(x,y)\in \gph F\cap \B_\delta(\xb,\yb)$ and all $(y^*,x^*)$ belonging to any $L\in\Sp^*F(x,y)$.
\end{definition}
 Clearly, every mapping with the SCD property around  $(\xb,\yb) \in\gph F$
which is \ssstar at $(\xb,\yb)$ is automatically SCD \ssstar at  $(\xb,\yb)$. Therefore, the class of \SCD \ssstar mappings is even richer than the class of \ssstar maps. In particular, it follows from \cite[Theorem 2]{Jou07} that every mapping whose graph is a closed subanalytic set is \SCD \ssstar, cf. \cite{GfrOut22a}.

The following proposition provides the key estimate for the \ssstar Newton method for \SCD mappings.
\begin{proposition}[{\cite[Proposition 5.3]{GfrOut22a}}]\label{PropConvNewton}
  Assume that $F:\R^n\tto\R^n$ is \SCD \ssstar at $(\xb,\yb)\in\gph F$. Then for every  $\epsilon>0$ there is some $\delta>0$ such that the estimate
  \begin{equation*}%\label{EqBndNewtonStep}
  \norm{x-C_L^T(y-\yb)-\xb}\leq \epsilon\sqrt{n(1+\norm{C_L}^2)}\norm{(x,y)-(\xb,\yb)}
  \end{equation*}
  holds for every $(x,y)\in\gph F\cap \B_\delta(\xb,\yb)$ and every $L\in\Sp^*F(x,y)\cap\Z_n^{\rm reg}$.
\end{proposition}
We now describe the \SCD variant of the \ssstar Newton method. Given a solution $\xb\in F^{-1}(0)$ of \eqref{EqIncl} and some positive scalar, we define the mappings $\A_{\eta,\xb}:\R^n\tto\R^n\times\R^n$ and $\New_{\eta,\xb}:\R^n\tto\R^n$ by
\begin{gather*}
  \A_{\eta,\xb}(x):=\{(\hat x,\hat y)\in\gph F\mv \norm{(\hat x,\hat y)-(\xb,0)}\leq \eta\norm{x-\xb}\},\\
  \New_{\eta,\xb}(x):=\{\hat x-C_L^T\hat y\mv (\hat x,\hat y)\in \A_{\eta,\xb}(x), L\in\Sp^*F(\hat x,\hat y)\cap \Z_n^{\rm reg}\}.
\end{gather*}

\begin{proposition}[{\cite[Proposition 4.3]{GfrOutVal21}}]\label{PropSingleStep}
  Assume that $F$ is \SCD \ssstar in $(\xb,0) \in\gph F$ and \SCD regular around $(\xb,0)$ and let $\eta>0$. Then there is some $\bar\delta>0$  such that for every $x\in \B_{\bar\delta}(\xb)$ the mapping $F$ is \SCD regular around every point $(\hat x,\hat y)\in \A_{\eta,\xb}(x)$. Furthermore, for every $\epsilon>0$ there is some $\delta\in(0,\bar\delta]$ such that
  \[\norm{z-\xb}\leq\epsilon\norm{x-\xb}\ \forall x\in \B_\delta(\xb), \forall z\in \New_{\eta,\xb}(x).\]
\end{proposition}
Assuming that we are given some iterate $x^{(k)}$, the next iterate is given formally by $\ee x{k+1}\in \New_{\eta,\xb}(\ee xk)$. Let us take a closer look at this rule. Since we are dealing with set-valued mappings $F$, we cannot expect, in general, that $F(x^{(k)})\not=\emptyset$ or that $0$ is close to
$F(x^{(k)})$, even if $x^{(k)}$ is close to a solution $\xb$. Therefore, we first perform some step that produces $(\hat x^{(k)},\hat y^{(k)})\in\gph F$ as an approximate projection of $(x^{(k)},0)$ onto $\gph F$. We require that
\begin{equation}\label{EqBndApprStep}
\norm{(\hat x^{(k)},\hat y^{(k)})-(\xb,0)}\leq \eta\norm{x^{(k)}-\xb}
\end{equation}
for some constant $\eta>0$, i.e., $(\ee{\hat x}k,\ee{\hat y}k)\in\A_{\eta,\xb}(\ee xk)$. For instance, if
\[\norm{(\hat x^{(k)},\hat y^{(k)})-(x^{(k)},0)}\leq \beta\dist{(x^{(k)},0),\gph F} \] is true with some $\beta\geq 1$, then
\begin{align*}\norm{(\hat x^{(k)},\hat y^{(k)})-(\xb,0)}&\leq \norm{(\hat x^{(k)},\hat y^{(k)})-(x^{(k)},0)}+\norm{(x^{(k)},0)-(\xb,0)}\\
&\leq  \beta\dist{(x^{(k)},0),\gph F}+\norm{(x^{(k)},0)-(\xb,0)}\leq (\beta+1)\norm{(x^{(k)},0)-(\xb,0)}
\end{align*}
and thus \eqref{EqBndApprStep} holds with $\eta=\beta+1$ and we can fulfill the inequality \eqref{EqBndApprStep} without knowing the solution $\xb$. Furthermore, we require that $\Sp^*F(\hat x^{(k)},\hat y^{(k)})\cap\Z_n^{\rm reg}\not=\emptyset$ and compute the new iterate as $x^{(k+1)}=\hat x^{(k)}-C_L^T\hat y^{(k)}$ for some $L\in \Sp^*F(\hat x^{(k)},\hat y^{(k)})\cap\Z_n^{\rm reg}$. In fact, in our numerical implementation we will not compute the matrix $C_L$, but two $n\times n$ matrices $A,B$ such that $L=\rge(B^T,A^T)$. The next iterate $x^{(k+1)}$ is then obtained by  $x^{(k+1)}=\hat x^{(k)}+\Delta x^{(k)},$ where  $\Delta x^{(k)}$ is a solution of the system $A\Delta x=-B\hat y^{(k)}$. Alternatively, in view of Proposition \ref{PropC_L}, we can also choose a subspace $L\in \Sp F(\hat x^{(k)},\hat y^{(k)})\cap\Z_n^{\rm reg}$ and compute the Newton direction as $\ee{\Delta x}k=-C_L\ee{\hat y}k$, that is, given $(A,B)\in \M(L)$ we have
$\ee{\Delta x}k=-Ap$ where $p$ solves $Bp=\ee yk$.

This leads to the following conceptual algorithm.
\begin{algorithm}[\SCD \ssstar Newton-type method for inclusions]\label{AlgNewton}\mbox{ }\\
 1. Choose a starting point $x^{(0)}$, set the iteration counter $k:=0$.\\
 2. If ~ $0\in F(x^{(k)})$, stop the algorithm.\\
  3. \begin{minipage}[t]{\myAlgBox} {\bf Approximation step: } Compute
  $$(\hat x^{(k)},\hat y^{(k)})\in\gph F$$ satisfying \eqref{EqBndApprStep} and such that $\Sp^*F(\hat x^{(k)},\hat y^{(k)})\cap\Z_n^{\rm reg}\not=\emptyset$.\end{minipage}\\
  4. \begin{minipage}[t]{\myAlgBox} {\bf Newton step: }Compute the Newton direction $\ee{\Delta x}k$ by one of the following two alternatives:
  \begin{enumerate}
  \item[a)]
  Select $n\times n$ matrices $A^{(k)},B^{(k)}$ with
  $$L^{(k)}:=\rge\big({B^{(k)}}^T,{A^{(k)}}^T)\in \Sp^*F(\hat x^{(k)},\hat y^{(k)})\cap\Z_n^{\rm reg}$$ and calculate the Newton direction $\Delta x^{(k)}$ as a solution of the linear system $$A^{(k)}\Delta x=-B^{(k)}\hat y^{(k)}.$$
  \item[b)]Select $n\times n$ matrices $A^{(k)},B^{(k)}$ with
$$L^{(k)}:= \rge\big({A^{(k)}},{B^{(k)}})\in \Sp F(\hat x^{(k)},\hat y^{(k)})\cap\Z_n^{\rm reg},$$ compute a solution $p$ of the linear system
$${B^{(k)}}p =-\hat y^{(k)}$$  and obtain the  Newton direction as $\Delta x^{(k)}=A^{(k)}p$.
  \end{enumerate}
   Compute the new iterate via $x^{(k+1)}=\hat x^{(k)}+\Delta x^{(k)}.$\end{minipage}\\
  \strut5. Set $k:=k+1$ and go to 2.
\end{algorithm}
For the choice between the two approaches to calculate the Newton direction, it is important to consider whether an element from $\Sp^*F(\hat x^{(k)} ,\hat y^{(k)})$ or from $\Sp F(\hat x^{(k)} ,\hat y^{(k)})$ is easier to compute.

For this algorithm, locally superlinear convergence follows from Proposition \ref{PropSingleStep}, see also \cite[Corollary 5.6]{GfrOut22a}.
\begin{theorem}\label{ThConvSSNewton}
 Assume that $F$ is \SCD \ssstar at $(\xb,0) \in\gph F$ and \SCD regular around $(\xb,0)$. Then for every $\eta>0$ there is a neighborhood $U$ of $\xb$ such that
 for every starting point $x^{(0)}\in U$ Algorithm \ref{AlgNewton} is well defined and stops after finitely many iterations at a solution of \eqref{EqIncl} or produces a sequence $x^{(k)}$ that superlinearly converges to $\xb$ for any choice of $(\hat x^{(k)} ,\hat y^{(k)})$ satisfying \eqref{EqBndApprStep} and any $L^{(k)}\in\Sp^*F(\hat x^{(k)} ,\hat y^{(k)})$ in Step 4.a) and any $L^{(k)}\in\Sp F(\hat x^{(k)} ,\hat y^{(k)})$ in Step 4.b).
\end{theorem}

In particular, if $F$ is strongly metrically regular around $(\xb,0)$, then $F$ has the \SCD property around $(\xb,0)$ by \cite[Corollary 3.19]{GfrOut22a} and it is also \SCD regular around $(\xb,0)$ as pointed out in the previous section. Therefore, if $F$ also happens to be SCD \ssstar around $(\xb,0)$, then the assumptions of the above statement are fulfilled.

Note that for an implementation of the Newton step, we need not know the whole derivative $\Sp^*F(\hat x^{(k)},\hat y^{(k)})$ (or $\Sp F(\hat x^{(k)},\hat y^{(k)})$) but only one element $L^{(k)}$.

\section{On the implementation of the SCD semismooth* Newton method}

When trying to implement the \SCD \ssstar Newton method directly for \eqref{EqGE}, it turns out that it can be rather difficult to perform the approximation step. Hence, we consider another equivalent approach which is more flexible. Consider an equivalent reformulation of \eqref{EqGE} by the (decoupled) GE
\begin{equation}\label{EqGEMod}
  0\in G(x,d)=\myvec{f(x)+Q(d)\\x-d}
\end{equation}
in variables $(x,d)\in \R^n\times\R^n$. Obviously, $0\in H(\xb)$ holds if and only if $(0,0)\in G(\xb,\xb)$.

The new variable $d$ acts only as an auxiliary variable. The approximation step now reads as follows: Given $\ee xk$ close to a solution $\xb$ (and arbitrary $\ee dk$, for example, $\ee dk=\ee xk$), set $\ee{\hat x}k:=\ee xk$ and find a point $\ee{\hat d}k$ close to $\ee xk$ such that
$\dist{0,f(\ee xk)+Q(\ee{\hat d}k)}$ is small. An approach to solving this problem could be to rewrite GE \eqref{EqGE} in fixed point form $x\in T(x)$ and select $\ee{\hat d}k\in T(\ee xk)$.
For example, for any $\lambda>0$ there is
\begin{equation}\label{EqFixPoint}\eqref{EqGE}\ \Leftrightarrow\ x-\lambda f(x)\in (I+\lambda Q)(x)\ \Leftrightarrow\ x\in(I+\lambda Q)^{-1}\big(x-\lambda f(x)\big).\end{equation}
If we choose $\ee{\hat d}k\in (I+\lambda Q)^{-1}\big(\ee xk-\lambda f(\ee xk)\big)$, we have $\ee xk-\lambda f(\ee xk)\in \ee{\hat d}k+\lambda Q(\ee{\hat d} k)$ and
\begin{equation}\label{EqResGen}\ee{\hat y}k:=\big(\frac 1\lambda(\ee xk-\ee{\hat d}k),\ee xk-\ee{\hat d}k\big)\in G(\ee xk,\ee{\hat d}k)\end{equation}
follows. In order to show that this approach is feasible as an approximation step, we have to verify that a bound of the form \eqref{EqBndApprStep} holds, at least for $\ee xk$ close to $\xb$.
\begin{proposition}\label{PropApprStepResolvent}Let $\xb$ be a solution of \eqref{EqGE} and assume that there is some $\lambda>0$ such that the resolvent $(I+\lambda Q)^{-1}$ has a single-valued Lipschitz continuous localization $S$ around $\xb-\lambda f(\xb)$ for $\xb$, i.e., there are neighborhoods $V$ of $\xb-\lambda f(\xb)$ and $U$ of $\xb$ such that $S:V\to U$ is Lipschitz continuous and $(I+\lambda Q)^{-1}(z)\cap U =\{S(z)\}$ for $z\in V$. Then there is some $\delta>0$  and some $\eta>0$ such that for every $x\in \B_\delta(\xb)$ there holds $x-\lambda f(x)\in V$ and  the vectors $\hat d:=S(x-\lambda f(x))$, $\hat y:=(\frac 1\lambda (x-\hat d), x-\hat d)\in\gph G(x,\hat d)$ satisfy the estimate
\[\norm{\big((x,\hat d), \hat y\big)-\big((\xb,\xb),0\big)}\leq\eta\norm{x-\xb}.\]
\end{proposition}
\begin{proof}
  Choose $\delta>0$ such that $\B_\delta(\xb)\subseteq (I-\lambda f)^{-1}(V)$ and let $L_f, L_S>0$ denote the  moduli of Lipschitz continuity of $f$ on $\B_\delta(\xb)$ and of $S$ on $V$, respectively. Consider $x\in \B_\delta(\xb)$. Then, by construction, $x-\lambda f(x)\in V$ and hence $\hat d:=S(x-\lambda f(x))$ is well defined. Further, by \eqref{EqFixPoint}, we have $\xb=S(\xb-\lambda f(\xb))$ implying
  \[\norm{\hat d-\xb}=\norm{S(x-\lambda f(x))-S(\xb-\lambda f(\xb))}\leq L_S\norm{x-\lambda f(x)-(\xb-\lambda f(\xb))}\leq L_S(1+\lambda L_f)\norm{x-\xb}.\]
  In addition we have $x-\lambda f(x)\in \hat d+\lambda Q(\hat d)$ and $\hat y\in\gph G(x,\hat d)$ follows. Since $\norm{\hat y}\leq(1+1/\lambda)\norm{\hat d-x}\leq
  (1+1/\lambda)(\norm{\hat d-\xb}+\norm{x-\xb})$, we obtain
  \begin{align*}\norm{\big((x,\hat d), \hat y\big)-\big((\xb,\xb),0\big)}&\leq \norm{x-\xb}+\norm{\hat d-\xb}+\norm{\hat y}\leq (2+\frac 1\lambda)(\norm{x-\xb}+\norm{\hat d-\xb})\\
  &\leq (2+\frac 1\lambda)(1+L_S(1+\lambda L_f))\norm{x-\xb}\end{align*}
  and the assertion follows.
\end{proof}
In particular, if $Q$ is a maximal hypomonotone mapping, i.e., there exists some $\gamma\geq 0$ such that $\gamma I+Q$ is maximal monotone, then for every $0<\lambda<1/\gamma$ the mapping $(I+\lambda Q)$ is strongly monotone and hence $(I+\lambda Q)^{-1}$ is a single-valued Lipschitz continuous function on $\R^n$, cf. \cite[Proposition 12.54]{RoWe98}. However, hypomonotonicity is only a sufficient condition ensuring  that $(I+\lambda Q)^{-1}$ has this property. In Section \ref{SecAlgContactProbl} we will encounter a non-hypomonotone mapping $\tilde Q$, such that $(I+\lambda\tilde Q)^{-1}$ is single-valued and Lipschitz continuous for every $\lambda>0$.

Note that the choice $\hat d\in (I+\lambda Q)^{-1}\big(x-\lambda f(x)\big)$ corresponds to one step of the so-called {\em Forward-Backward method} for solving \eqref{EqGE}.

In the next proposition, we summarize some properties of $G$.
\begin{proposition}\label{PropG}
\begin{enumerate}
  \item[(i)] For every $x\in\R^n$ and $(d,z)\in\gph Q$ we have
  \begin{align*}\Sp G\big((x,d),(f(x)+z,x-d)\big)&=\left\{\rge\left[\left(\begin{matrix}I&0\\0&X\end{matrix}\right), \left(\begin{matrix}\nabla f(x)&Y\\I&-X\end{matrix}\right)\right]\bbmv \rge(X,Y)\in\Sp Q(d,z)\right\},\\
  \Sp^* G\big((x,d),(f(x)+z,x-d)\big)
    &=\left\{\rge\left[\left(\begin{matrix}Y^*&0\\0&I\end{matrix}\right), \left(\begin{matrix}\nabla f(x)^TY^*&I\\X^*&-I\end{matrix}\right)\right]\bbmv \rge(Y^*,X^*)\in\Sp^* Q(d,z)\right\}.
  \end{align*}
  \item[(ii)] Let $x\in\R^n$ and assume that $Q$ has the \SCD property around $(d,z)\in \gph Q$. Then the following statements are equivalent:
  \begin{enumerate}
    \item[(a)] $G$ is \SCD regular around $\big((x,d),( f(x)+z,x-d)\big)$.
    \item[(b)] For every $L\in \Sp^* Q(d,z)$  and every  $(X,Y)\in\M(L)$ the matrix $\nabla f(x)X+Y$ is nonsingular.
    \item[(c)]  For every $L\in \Sp^* Q(d,z)$  and every  $(Y^*,X^*)\in\M(L)$ the matrix $\nabla f(x)^TY^*+X^*$ is nonsingular.
  \end{enumerate}
  \item[(iii)] The mapping $H$ is \SCD regular around $(\xb,0)$ if and only if $G$ is \SCD regular around $\big((\xb,\xb),(0,0))$.
\end{enumerate}
\end{proposition}
\begin{proof}
$G$ has the representation $G(x,d)=h(x,d)+\tilde Q(x,d)$ with
\[h(x,d):=\myvec{f(x)\\x-d}\mbox{ and }\tilde Q(x,d)=\myvec{Q(d)\\0}.\]
Since $\gph D\tilde Q\big((x,d),(z,0)\big)=\{\big((u,e), (v,0)\big)\mv (e,v)\in \gph DQ(d,z)\}$, we obtain $\OO_{\tilde Q}=\R^n\times\OO_Q\times \{0\}$. For every $(d,z)\in\OO_Q$ and every $x\in\R^n$ the orthogonal projection onto $\tilde L:=\gph D\tilde Q\big((x,d),(z,0)\big)=\{\big((u,e), (v,0)\big)\mv (e,v)\in \gph DQ(d,z)\}$ is represented by the matrix
\[P_{\tilde L}=\left(\begin{matrix}I&0&0\\0&P_L&0\\0&0&0\end{matrix}\right),\]
where $P_L$ corresponds to the orthogonal projection onto the subspace $L:=\gph DQ(d,z)$. Hence, for every $(d,z)\in\gph Q$ and every $x\in\R^n$ we obtain
\begin{align*}\Sp\tilde Q\big((x,d), (z,0)\big)&=\{\R^n\times L\times\{0\}\mv L\in \Sp Q(d,z)\}\\
&=\left\{\rge\left[\left(\begin{matrix}I&0\\0&X\end{matrix}\right), \left(\begin{matrix}0&Y\\0&0\end{matrix}\right)
\right]\bbmv \rge(X,Y)\in\Sp Q(d,z)
\right\}\end{align*}
and from Proposition \ref{PropSum} we conclude
\begin{align*}\Sp G\big((x,d),(f(x)+z,x-d)\big)&=\left\{\rge\left[\left(\begin{matrix}I&0&0&0\\0&I&0&0\\\nabla f(x)&0&I&0\\I&-I&0&I\end{matrix}\right)\left(\begin{matrix}I&0\\0&X\\0&Y\\0&0\end{matrix}\right)\right]\bbmv \rge(X,Y)\in\Sp Q(d,z)\right\}\\
&=\left\{\rge\left[\left(\begin{matrix}I&0\\0&X\end{matrix}\right), \left(\begin{matrix}\nabla f(x)&Y\\I&-X\end{matrix}\right)\right]\bbmv \rge(X,Y)\in\Sp Q(d,z)\right\}.\end{align*}
Similarly, we have
\begin{align*}\Sp^*\tilde Q\big((x,d),(z,0)\big)&=\left\{\rge\left[\left(\begin{matrix}Y^*&0\\0&I\end{matrix}\right), \left(\begin{matrix}0&0\\X^*&0\end{matrix}\right)\right]
\bbmv \rge(Y^*,X^*)\in\Sp^* Q(d,z)\right\},
\end{align*}
yielding, together with Proposition \ref{PropSum},
\begin{align*}\lefteqn{\Sp^* G\big((x,d),(f(x)+z,x-d)\big)}\\
&=\left\{\rge\left[\left(\begin{matrix}I&0&0&0\\0&I&0&0\\\nabla f(x)^T&I&I&0\\0&-I&0&I\end{matrix}\right)\left(\begin{matrix}Y^*&0\\0&I\\0&0\\X^*&0\end{matrix}\right)\right]\bbmv \rge(Y^*,X^*)\in\Sp^* Q(d,z)\right\}\\
&=\left\{\rge\left[\left(\begin{matrix}Y^*&0\\0&I\end{matrix}\right), \left(\begin{matrix}\nabla f(x)^TY^*&I\\X^*&-I\end{matrix}\right)\right]\bbmv \rge(Y^*,X^*)\in\Sp^* Q(d,z)\right\}.\end{align*}

By virtue of (i) and the definition of \SCD regularity, $G$ is \SCD regular around $\big((x,d),(f(x)+z,x-d)\big)$ if and only if for every pair $X,Y$ with $\rge(X,Y)\in \Sp Q(d,z)$ the matrix
\[\left(\begin{matrix}\nabla f(x)&Y\\I&-X\end{matrix}\right)=\left(\begin{matrix}\nabla f(x)&\nabla f(x)X+Y\\I&0\end{matrix}\right)\left(\begin{matrix}I&-X\\0&I\end{matrix}\right)\]
is nonsingular. By the representation above, this holds if and only if $\nabla f(x)X+Y$ is nonsingular. Thus the equivalence between (a) and (b) is established. Similarly, $G$ is \SCD regular around $\big((x,d),(f(x)+z,x-d)\big)$ if and only if for every pair $Y^*,X^*$ with $\rge(Y^*,X^*)\in \Sp^* Q(d,z)$ the matrix
\[\left(\begin{matrix}\nabla f(x)^TY^*&I\\X^*&-I\end{matrix}\right)= \left(\begin{matrix}I&-I\\0&I\end{matrix}\right)\left(\begin{matrix}\nabla f(x)^TY^*+X^*&0\\X^*&-I\end{matrix}\right)\]
is nonsingular and the equivalence between (a) and (c) follows.

To establish (iii), just note that by Proposition \ref{PropSum} we have
\[\Sp H(\xb,0)=\left\{\rge\left[\left(\begin{matrix}I&0\\\nabla f(\xb)&I\end{matrix}\right)\myvec{X\\Y}\right]=\rge(X,\nabla f(x)X+Y)\bmv \rge(X,Y)\in\Sp Q(\xb,-f(\xb))\right\}\]
and the assertion follows from (ii) and the definition of \SCD regularity.
\end{proof}

Let us now consider the Newton step. Assume that, emanating from the iterate $\ee xk$, we have computed $\big((\ee{\hat x}k,\ee{\hat d}k),(\ee{\hat y_1}k, \ee{\hat y_2}k)\big)\in\gph G$ as the result of the approximation step.

{\bf Case (i):} We compute the Newton direction according to step 4.a) of Algorithm \ref{AlgNewton}.\\ By Proposition \ref{PropG}, we have to compute two $n\times n$ matrices $\ee {{Y^*}}k, \ee {X^*}k$ with $\rge(\ee {Y^*}k, \ee {X^*}k)\in\Sp^*Q(\ee{\hat d}k,\ee{\hat y_1}k-\nabla f(\ee{\hat x}k))$ and to solve the system
\[\left(\begin{matrix}\nabla f(\ee{\hat x}k)^T\ee{Y^*}k&I\\\ee{X^*}k&-I\end{matrix}\right)^T\myvec{\ee{\Delta x}k\\\ee{\Delta d}k}=\left(\begin{matrix}{\ee{Y^*}k}^T\nabla f(\ee{\hat x}k)&{\ee{X^*}k}^T\\I&-I\end{matrix}\right)\myvec{\ee{\Delta x}k\\\ee{\Delta d}k}=-\left(\begin{matrix}\ee{Y^*}k&0\\0&I\end{matrix}\right)^T\myvec{\ee {\hat y_1}k\\\ee {\hat y_2}k}.\]
Using the second equation we can eliminate $\ee{\Delta d}k=\ee{\Delta x}k+\ee{\hat y_2}k$ and arrive at the linear system
\begin{equation}\label{EqNewtonDir}\left({\ee{Y^*}k}^T\nabla f(\ee{\hat x}k)+{\ee{X^*}k}^T\right)\ee{\Delta x}k=-\left({\ee{Y^*}k}^T\ee{\hat y_1}k+{\ee{X^*}k}^T\ee{\hat y_2}k\right).\end{equation}

{\bf Case (ii):} The Newton direction is computed by  step 4.b) of Algorithm \ref{AlgNewton}.\\
 In this case we determine two $n\times n$ matrices $\ee Xk,\ee Yk$ with $\rge(\ee Xk,\ee Yk)\in \Sp Q(\ee{\hat d}k,\ee{\hat y_1}k-\nabla f(\ee{\hat x}k))$, solve the linear system
\[ \left(\begin{matrix}\nabla f(\ee{\hat x}k)&\ee Yk\\I&-\ee Xk\end{matrix}\right)\myvec{p_1\\p_2}=-\myvec{\ee {\hat y_1}k\\\ee {\hat y_2}k}\]
and  set
\[\myvec{\ee{\Delta x}k\\\ee{\Delta d}k}=\left(\begin{matrix}I&0\\0&\ee Xk\end{matrix}\right)\myvec{p_1\\p_2}.\]
By eliminating $p_1=\ee Xk p_2-\ee{\hat y_2}k$ we obtain the linear system
\[\left(\nabla f(\ee{\hat x}k)\ee Xk+ \ee Yk\right)p_2=\nabla f(\ee{\hat x}k)\ee{\hat y_2}k-\ee{\hat y_1}k,\]
 whose solution yields
 \[\myvec{\ee{\Delta x}k\\\ee{\Delta d}k}=\myvec{\ee Xk p_2-\ee{\hat y_2}k\\\ee Xk p_2}.\]

 In both cases, the new iterate is given by $\ee x{k+1}:=\ee{\hat x}k+\ee{\Delta x}k$. Further, we have $\ee {\Delta x}k-\ee{\Delta d}k=-\ee{\hat y_2}k=\ee{\hat d}k-\ee{\hat x}k$ resulting in $\ee x{k+1}=\ee{\hat d}k+\ee{\Delta d}k$.

\section{Algebraic form of the discrete contact problem with Coulomb friction\label{SecAlgContactProbl}}
We consider an elastic body represented by a  bounded domain $\Omega\subset\R^3$ with a sufficiently smooth boundary $\partial\Omega$. The body is made of elastic, homogeneous, and isotropic material. The boundary consists of three non-empty disjoint parts: $\partial\Omega=\overline{\Gamma_u}\cup \overline{\Gamma_p}\cup \overline{\Gamma_c}$. Zero displacements are prescribed on $\Gamma_u$, surface tractions act on $\Gamma_p$, and the body is subject to volume forces. We seek a displacement field and a corresponding stress field satisfying the Lam\'e system of PDEs in $\Omega$, the homogeneous Dirichlet boundary conditions on $\Gamma_u$, and the Neumann boundary conditions on $\Gamma_p$. The body is unilaterally supported along $\Gamma_c$ by some flat rigid foundation given by the halfspace $\R^2\times\R_-$ and the initial gap between the body and the rigid foundation is denoted by $d(x)$, $x\in\Gamma_c$.  In the contact zone, we consider a static Coulomb Friction condition.

This problem can be described by partial differential equations and boundary conditions for the displacements, which we are looking for. We refer the reader to, e.g., \cite{EJ}, where also a weak formulation can be found. We consider here only the discrete algebraic problem, which arises after some suitable finite element approximation.

Let $n$ denote the number of degrees of freedom of the nodal displacement vector and let $p$ denote the number of contact nodes $x_i\in \overline{\Gamma_c}\setminus\overline{\Gamma_u}$. After some suitable reordering of the variables, such that the first $3p$ positions are occupied by the displacements of the nodes lying in the contact part of the boundary, we arrive at the following nodal block structure for an arbitrary vector $y\in\R^n$:
\[y=(y^1,\ldots,y^p,y^R)\quad \mbox{with}\quad y^i\in\R^3,\ i=1,\ldots,p,\ y^R\in\R^{n-3p}.\]
In what follows, $A\in\R^{n\times n}$, $\tilde l\in\R^n$ are the stiffness matrix and the load vector, respectively. Further we are given two matrices $N\in\R^{p\times n}$, $T\in\R^{2p\times n}$, where, for a given displacement vector $v$, $Nv$ yields the normal components at the $p$ contact points, and $Tv=(T^1v,\ldots,T^pv)$, where $T^iv\in\R^2$ is the  tangential nodal displacement vector at the $i$-th contact node. The symbol $\vert Tv\vert \in\R^p$ denotes a
vector defined by
\[\vert Tv\vert=(\norm{T^1v},\ldots,\norm{T^pv}).\]
We denote with $\alpha\in\R^p$ the vector of nodal distances with $\alpha_i:=d(x_i)$ and the friction coefficient is denoted by $\F$.
\begin{definition}[{\cite[Definition 3.6]{BeHaKoKuOut09}}]\label{DefContProbl}
As a solution of a discrete contact problem with Coulomb friction
we declare any couple $(\tilde u, \lambda)\in\R^n\times\R^p_+$ satisfying
\begin{align}
\label{EqDefSol1}  &\skalp{A\tilde u,v-\tilde u}+\F\skalp{\lambda,\vert Tv\vert -\vert T\tilde u\vert}\geq \skalp{\tilde l,v-\tilde u}+\skalp{\lambda, Nv-N\tilde u}\ \forall v\in\R^n,\\
\label{EqDefSol2}  &\skalp{\mu-\lambda, N\tilde u+\alpha}\geq 0,\ \forall \mu\in \R^p_+.
\end{align}
\end{definition}
Since the stiffness matrix $A$ is positive definite and $\lambda\geq0$, condition \eqref{EqDefSol1} is equivalent to the requirement that $\tilde u$ is a minimizer of the convex minimization problem
\begin{equation*}\min_v\frac 12v^T Av-\skalp{\tilde l,v}-\skalp{\lambda, Nv}+\F\skalp{\lambda,\vert Tv\vert}.\end{equation*}
Given a vector $z=(z_1,z_2,z_3)^T\in\R^3$, we denote by $z_{12}:=(z_1,z_2)^T\in\R^2$ the vector formed by the first two components. Using this notation, we have
\[Nv=(v_3^1,\ldots,v_3^p)^T\quad\mbox{and}\quad T^iv=v^i_{12},\ i=1,\ldots,p\]
due to the ordering of the nodal displacements.

Next consider the transformation of variables $u=\tilde u+d$, where $d=(d^1,\ldots,d^p,d^R)^T\in\R^n$ is given by
\[d^i_{12}:=0,\ d^i_3:=\alpha_i,\ i=1,\ldots,p,\quad d^R:=0.\]
Then $u$ is a solution of the problem
\[\min_v\frac 12v^T Av-\skalp{l,v}-\sum_{i=1}^p\lambda_iv^i_3+\sum_{i=1}^p\F\lambda_i\norm{v^i_{12}},\]
where $l:=\tilde l-Ad$. Since the objective in this minimization problem is convex, this is in turn equivalent to the first-order optimality condition
\begin{align*}0&\in (Au-l)^i-\lambda_i(0,0,1)^T+\F\lambda_i\partial\norm{u^i_{12}},\ i=1,\ldots,p,\\
0&=(Au-l)^R.
\end{align*}
Further, \eqref{EqDefSol2} is the same as $-\lambda_i\in N_{\R_+}(\tilde u^i_3+\alpha_i)=N_{\R_+}(u^i_3)$, $i=1,\ldots,p$. After eliminating $\lambda$ from explicit variables, we have thus arrived at the GE
\begin{equation}\label{eq-8}
0 \in H(u):=Au-l+Q(u),
\end{equation}
%%%
where
\begin{equation}\label{eq-9}
Q(u)=\prod^{p}_{i=1} \tilde Q(u^{i})\times Q^R(u^R)
\end{equation}
%%%
with $\tilde Q:\R^3\tto\R^3$ and $Q^R:\R^{n-3p}\tto\R^{n-3p}$ defined by
%%%

\begin{equation}\label{eq-10}
 \tilde Q(v):= \left\{\myvec{-\F\vartheta\partial\norm{v_{12}}\\ \vartheta}\,\bigg|\, \vartheta\in N_{\R_+}(v_3)\right\}\quad \mbox{and}\quad Q^R(v):=\{0\}.
\end{equation}
 GE \eqref{eq-8} is dependent solely on transformed displacements $u$. Multipliers $\lambda_i$ appear only implicitly as $-\vartheta$ in the description of $\tilde Q$. This is a big difference with respect to other approaches, where the semismooth Newton method for equations is applied to mixed primal-dual systems or purely dual systems using some NCP-functions, see, e.g., \cite{Wohl11},\cite{BlFrFrRa16},\cite{OKZ}.

\begin{remark}
  We have derived the GE \eqref{eq-8} for the Signorini problem with static Coulomb friction. We claim also that, for other contact problems with friction involving two elastic bodies, one can derive a GE of the same type. The interested reader is referred to \cite[Section 5.2]{Wohl11} for an algebraic transformation of a two-body problem to a one-body problem.
\end{remark}

Note that
\[\gph \tilde Q=\left\{(v,g,\vartheta)\in\R^3\times\R^2\times\R\mv g\in-\F\vartheta\partial\norm{v_{12}},\ \vartheta\in N_{\R_+}(v_3)\right\},\]
which enables us to prove the following statement.
\begin{proposition}
$H$ is \ssstar at each point in its graph.
\end{proposition}
\begin{proof}
Consider a point $(\bar{u}, \bar{w}) \in \gph H$. By \cite[Proposition 3.6]{GfrOut21} it suffices to show that $Q$ is \ssstar at $(\bar{u},\bar{w}-A\bar{u}-l)$,
which definitely holds true provided $\tilde Q$ is \ssstar at all points of its graph. Thus, invoking \cite[Theorem 3]{Jou07} and using the connection between the \ssstar property of  sets and the respective distance functions, it suffices to show that $\gph \tilde Q$ is a subanalytic set. Let us pick a reference point $(\bar{v},\bar{g},\bar{\vartheta}) \in \gph \tilde Q$ and consider the set
%%%
\[P=\left\{(v,g,\vartheta,p)\in \R^3\times\R^2\times\R\times\R^2\,\left|\;
\begin{aligned} &\norm{v-\bar v}^2+\norm{(g,\vartheta)-(\bar g,\bar\vartheta)}^2\leq 1,\ \norm{p}^2\leq 1,\\
&\norm{v_{12}}^2p_1^2=v_1^2, \norm{v_{12}}^2p_2^2=v_2^2,\\
&v_1p_1\geq 0,\ v_2p_2\geq 0,\\
&g_1=-\F\vartheta p_1,\ g_2=-\F\vartheta p_2,\\
&v_3\geq 0,\ \vartheta\leq 0,\ v_3\vartheta=0
\end{aligned}\right.\right\}.\]
%%%
Clearly, $P$ is semianalytic (as the intersection of finitely many polynomial equalities and inequalities, it is even semialgebraic) and compact. Moreover, by construction, $\gph \widetilde{Q} \cap \mathcal{B}_1(\bar{v},\bar{g},\bar{\vartheta})$ is the canonical projection of $P$ onto the space of variables $(v,g,\vartheta)$ and hence subanalytic, cp. \cite{BM88}. The proof is complete.
\end{proof}
\begin{proposition}\label{PropGraphLipTildeQ}
For every $\gamma>0$, the mapping $(\gamma I_3+\tilde Q)^{-1}:\R^3\tto\R^3$ is single-valued and Lipschitz continuous on $\R^3$.  In particular, $\tilde Q$ is graphically Lipschitzian of dimension $3$ at every point of its graph.
\end{proposition}
\begin{proof}
  We have $v\in (\gamma I_3+\tilde Q)^{-1}(w)$ if and only if
  \begin{align}\label{EqAuxTildeQ1}\gamma v_3+\vartheta =w_3\\
  \label{EqAuxTildeQ2}\gamma v_{12}+\F(-\vartheta) v_{12}^*=w_{12}
  \end{align}
  for some $\vartheta\in N_{\R_+}(v_3)$ and some $v_{12}^*\in \partial\norm{v_{12}}$. Since $\gamma I_1+N_{\R_+}$ is both maximal monotone and strongly monotone, $v_3$ and $\vartheta$ are uniquely given by
  \begin{equation}\label{EqAuxTildeQ3}v_3=(\gamma I_1+N_{\R_+})^{-1}(w_3)=\frac{\max\{w_3,0\}}\gamma,\quad \vartheta=w_3-\gamma v_3=\min\{w_3,0\}.\end{equation}
  For given $\vartheta\leq 0$ the mapping $\gamma I_2+\F(-\vartheta)\partial\norm{\cdot}$ is again maximal monotone and strongly monotone and thus $v_{12}$ is uniquely given by
  \[v_{12}=(\gamma I_2+\F(-\vartheta)\partial \norm{\cdot})^{-1}(w_{12})=\begin{cases}0&\mbox{if $\norm{w_{12}}\leq \F(-\vartheta)$,}\\
  \frac 1\gamma\left(1-\frac{\F(-\vartheta)}{\norm{w_{12}}}\right)w_{12}&\mbox{if $\norm{w_{12}}>\F(-\vartheta)$.}\end{cases}\]
  These arguments prove that $(\gamma I_3+\tilde Q)^{-1}$ is single-valued on $\R^3$ and there remains to show the Lipschitz continuity. Consider two points $w^j$, $j=1,2$, together with
  $(\gamma I_3+\tilde Q)^{-1}(w^j)=\{v^j\}$ and the corresponding $\vartheta^j\in N_{\R_+}(v_3^j)$, $v_{12}^{*j}\in \partial\norm{v_{12}^j}$, $j=1,2$, according to \eqref{EqAuxTildeQ1}, \eqref{EqAuxTildeQ2}. Then we deduce from \eqref{EqAuxTildeQ2} that
  \begin{align*}\skalp{w^1_{12}-w^2_{12}, v^1_{12}-v^2_{12}}&=\gamma\norm{v^1_{12}-v^2_{12}}^2+\F\skalp{-\vartheta^1 v_{12}^{*1}+\vartheta^2 v_{12}^{*2},v^1_{12}-v^2_{12}}\\
  &=
  \gamma\norm{v^1_{12}-v^2_{12}}^2+\F(-\vartheta^1)\skalp{v_{12}^{*1}-v_{12}^{*2},v^1_{12}-v^2_{12}}+\F(\vartheta^2-\vartheta^1)\skalp{v_{12}^{*2},v^1_{12}-v^2_{12}}\\
  &\geq \gamma\norm{v^1_{12}-v^2_{12}}^2-\F\vert \vartheta^1-\vartheta^2\vert\norm{v^1_{12}-v^2_{12}},\end{align*}
  where we have used the facts that $-\vartheta^1\geq 0$,  that the subdifferential mapping $\partial\norm{\cdot}$ is monotone and that $\norm{v_{12}^{*2}}\leq 1$. Since the functions $t\to\min\{t,0\}$ and $t\to\max\{t,0\}$ are  Lipschitz continuous on $\R$ with constant $1$, we obtain from \eqref{EqAuxTildeQ3} that $\vert\vartheta^1-\vartheta^2\vert\leq \vert w_3^1-w_3^2\vert$ yielding
  \[\gamma\norm{v^1_{12}-v^2_{12}}^2\leq \skalp{w^1_{12}-w^2_{12}, v^1_{12}-v^2_{12}}+\F\vert w_3^1-w_3^2\vert\norm{v^1_{12}-v^2_{12}}\leq \big(\norm{w^1_{12}-w^2_{12}}+\F\vert w_3^1-w_3^2\vert\big)\norm{v^1_{12}-v^2_{12}}\]
  and consequently $\gamma\norm{v^1_{12}-v^2_{12}}\leq \norm{w^1_{12}-w^2_{12}}+\F\vert w_3^1-w_3^2\vert$. Since we also have $\gamma\vert v_3^1-v_3^2\vert\leq \vert w_3^1-w_3^2\vert$, we obtain
  \begin{align*}\gamma^2\norm{v^1-v^2}^2&\leq (\norm{w^1_{12}-w^2_{12}}+\F\vert w_3^1-w_3^2\vert)^2+\vert w_3^1-w_3^2\vert^2\leq 2(1+\F^2)(\norm{w^1_{12}-w^2_{12}}^2+\vert w_3^1-w_3^2\vert^2)\\
  &=2(1+\F^2)\norm{w^1-w^2}^2\end{align*}
  establishing Lipschitz continuity of $(\gamma I_3+\tilde Q)^{-1}$. To see that $\tilde Q$ is graphically Lipschitzian of dimension $3$, just take $\Phi(x,y)=(\gamma x+y,x)$ and $f:=(\gamma I+\tilde Q)^{-1}$ to obtain $\gph f=\Phi(\gph \tilde Q)$.
\end{proof}
\begin{remark}
  Note that the mapping $\tilde Q$ is not hypomonotone, that is, for every $\gamma>0$ the mapping $\gamma I_3+\tilde Q$ is not monotone. Indeed, consider $\gamma>0$ and let
  \[(v^1,w^1):=\big((1,0,0),(2\gamma,0,-2\frac\gamma\F)\big)\in\gph\tilde Q,\quad (v^2,w^2):=\big((2,0,0),(0,0,0)\big)\in \gph\tilde Q.\]
  Then
  \[\skalp{(\gamma v^1+w^1)-(\gamma v^2+w^2),v^1 -v^2}=\gamma\norm{v^1-v^2}^2+\skalp{(2\gamma,0,-2\frac\gamma\F),(-1,0,0)}=-\gamma<0\]
  and therefore  $\gamma I_3+\tilde Q$ is not monotone.\\
  Further note that in the case when $\vartheta=0$ and $v_{12}=0$ the subgradient $v_{12}^*\in\partial\norm{v_{12}}$ fulfilling \eqref{EqAuxTildeQ2} is not uniquely given.
\end{remark}
Throughout the sequel, it is convenient to refer to \cite{BeHaKoKuOut09} and express the graph of $\tilde Q$ in the form
\begin{equation}\label{eq-11}
\gph \tilde Q = L \cup M_{1} \cup M_{3}^{+} \cup M_{2} \cup M_{3}^{-} \cup M_{4},
 \end{equation}
%%%
where the single sets arising in \eqref{eq-11} do have a clear mechanical interpretation. Their definitions are provided in the following table.

\begin{table}[H]
\centering
\begin{tabular}{|l|c|c|c|}
\hline
& no contact &  weak contact & strong contact\\
& $v_{3}>0, \vartheta=0$  & $v_{3}=0, \vartheta=0$ &  $v_{3}=0, \vartheta<0$\\
\hline
sliding $v_{12}\neq 0$ & \multirow{ 3}{*}{$L$} & $M_{2}$ & $M_{1}$ \\
\cline{1-1} \cline{3-4}
%\hline
weak sticking & %\multirow{ 2}{*}{$L$}
&\multirow{ 2}{*}{$M_{4}$} & \multirow{ 2}{*}{$M_{3}^{-}$}  \\
$v_{12}=0,\|g\| = -\mathcal{F} \vartheta$ & &   & \\
\hline
strong sticking & \multirow{ 2}{*}{$-$} & \multirow{ 2}{*}{$-$} &
\multirow{ 2}{*}{$M^{+}_{3}$} \\
$v_{12}=0, \| g \|< -\mathcal{F} \vartheta$ & & & \\
\hline
\end{tabular}
\caption{Definitions and mechanic interpretations of the sets from \eqref{eq-11}.}\label{table_sets}
\end{table}
Note that in Table \ref{table_sets} the impossible combinations of variables are crossed out.

\begin{proposition}\label{prop:SCD}
$\tilde Q$ is an SCD mapping and $\mathcal{O}_{\tilde Q}= L \cup M_{1} \cup M_{3}^{+}$. In particular, for $(\bar{v},\bar{g},\bar{\vartheta}) \in L$
%%%
\begin{equation}\label{eq-12}
\widehat\Sp\tilde Q(\bar{v},\bar{g},\bar{\vartheta})= \Sp\tilde Q(\bar{v},\bar{g},\bar{\vartheta})=
\widehat\Sp^*\tilde Q(\bar{v},\bar{g},\bar{\vartheta})=\Sp^*\tilde Q(\bar{v},\bar{g},\bar{\vartheta})=\{\rge (I_3,0)\},
\end{equation}
%%%
for $(\bar{v},\bar{g},\bar{\vartheta})\in M_{3}^{+}$
%%%
\begin{equation}\label{eq-13}
\widehat\Sp\tilde Q(\bar{v},\bar{g},\bar{\vartheta})=\Sp\tilde Q(\bar{v},\bar{g},\bar{\vartheta})=
\widehat\Sp^*\tilde Q(\bar{v},\bar{g},\bar{\vartheta})=\Sp^*\tilde Q(\bar{v},\bar{g},\bar{\vartheta})=\{\rge (0,I_3)\},
\end{equation}
%%%
and for $(\bar{v},\bar{g},\bar{\vartheta})\in M_{1}$ one has
\begin{equation}\label{eq-14}
\widehat\Sp\tilde Q(\bar{v},\bar{g},\bar{\vartheta})=\Sp\tilde Q(\bar{v},\bar{g},\bar{\vartheta})=\left\{\rge  \left [
 \left (
\begin{array}{ll}
I_{2} & 0\\
0 & 0
\end{array}
\right ),
\left (
\begin{matrix}
-\mathcal{F}\bar{\vartheta}\frac{1}{\| \bar{v}_{12}\|}
\left ( I_2- \frac{\bar{v}_{12}\bar{v}_{12}^{T}}{\| \bar{v}_{12} \|^{2}} \right ) & -\mathcal{F}
\frac{\bar{v}_{12}}{\| \bar{v}_{12} \|}\\
0 & 1
\end{matrix}\right )
\right ]\right\},
\end{equation}
%%%
\begin{equation}\label{eq-15}
\widehat\Sp^*\tilde Q(\bar{v},\bar{g},\bar{\vartheta})=\Sp^*\tilde Q(\bar{v},\bar{g},\bar{\vartheta})=\left\{\rge  \left [
 \left (
\begin{matrix}
I_{2} & 0\\
\mathcal{F} \frac{\bar{v}_{12}^T}{\| \bar{v}_{12} \|} & 0
\end{matrix}
\right ),
\left (
\begin{matrix}
-\mathcal{F}\bar{\vartheta}\frac{1}{\| \bar{v}_{12}\|}
\left ( I_2- \frac{\bar{v}_{12}\bar{v}_{12}^{T}}{\| \bar{v}_{12} \|^{2}} \right ) & 0\\
0 & 1
\end{matrix}\right )
 \right ]\right\}.
\end{equation}
\end{proposition}
\begin{proof}Since $\tilde Q$ is graphically Lipschitzian of dimension 3 at every point of its graph, it is an \SCD mapping.

Note that the sets $L, M_1$ and $M_{3}^{+}$ exhibit a stable behavior in the sense that, for a sufficiently small $ \varrho > 0$,
$$
\left . \begin{array}{l}
(\bar{v},\bar{g},\bar{\vartheta})\in L( {\rm or}~ M_{1}, {\rm or}~ M_{3}^{+})\\
(v,g,\vartheta) \in \gph \tilde Q \cap \B_\rho(\bar{v},\bar{g},\bar{\vartheta})
\end{array}\right \} \Rightarrow (v,g,\vartheta) \in L ( {\rm or}~ M_{1}, {\rm or}~ M_{3}^{+}).
$$
In particular, we have
\[\gph\tilde Q\cap \B_\rho(\bar v,\bar g,\bar\vartheta)=\begin{cases}(\R^3\times\{0\})\cap \B_\rho(\bar v,\bar g,\bar\vartheta)&\mbox{if $(\bar v,\bar g,\bar\vartheta)\in L$,}\\
(\{0\} \times \R^3)\cap \B_\rho(\bar v,\bar g,\bar\vartheta)&\mbox{if $(\bar v,\bar g,\bar\vartheta)\in M_3^+$,}\\
\{(v_{12},0,-\F\vartheta\frac{v_{12}}{\norm{v_{12}}},\vartheta)\mv v_{12}\in\R^2,\vartheta\in\R\}\cap \B_\rho(\bar v,\bar g,\bar\vartheta)&\mbox{if $(\bar v,\bar g,\bar\vartheta)\in M_1$.}\end{cases}\]
It follows from Definition \ref{DefVarGeom} that
\begin{eqnarray*}\lefteqn{T_{\gph\tilde Q}(\bar v,\bar g,\bar\vartheta)=}\\
&\begin{cases}\R^3\times\{0\}&\mbox{if $(\bar v,\bar g,\bar\vartheta)\in L$,}\\
\{0\} \times \R^3&\mbox{if $(\bar v,\bar g,\bar\vartheta)\in M_3^+$,}\\
\left\{\left(h_{12},0,-\F\left(\bar\vartheta\frac1{\norm{\bar v_{12}}}\left(I-\frac{\bar v_{12}\bar v_{12}^T}{\norm{\bar v_{12}^2}}\right)h_{12}+\omega\frac{\bar v_{12}}{\norm{\bar v_{12}}}\right),\omega\right)\mv h_{12}\in\R^2,\omega\in\R\right\}&\mbox{if $(\bar v,\bar g,\bar\vartheta)\in M_1$.}\end{cases}
\end{eqnarray*}
In all three cases, we have therefore to do with linear subspaces of dimension three, which yield $\mathcal{O}_{\tilde Q} \supset L \cup M_{1} \cup M_{3}^{+}$ and the expressions for
$\widehat \Sp \tilde Q(\bar{v},\bar{g},\bar{\vartheta})$ in \eqref{eq-12}, \eqref{eq-13}, and \eqref{eq-14}.
Concerning the expressions for $\widehat\Sp^* \tilde Q(\bar{v},\bar{g},\bar{\vartheta})$, they can be derived by first computing the respective orthogonal complements and then using the relation \eqref{EqDualSubspace}. The equalities $\widehat\Sp \tilde Q(\bar v,\bar g,\bar\vartheta)= \Sp \tilde Q(\bar v,\bar g,\bar\vartheta)$ and $\widehat\Sp^* \tilde Q(\bar v,\bar g,\bar\vartheta)=\Sp^* \tilde Q(\bar v,\bar g,\bar\vartheta)$ in \eqref{eq-12}--\eqref{eq-15} follow from the observation that the matrices that describe the subspaces continuously depend on the argument $(\bar v,\bar g,\bar\vartheta)$.

It remains to show that actually $\OO_{\tilde Q} = L \cup M_{1} \cup M_{3}^{+}$, that is, $(M_2\cup M_4\cup M_3^-)\cap\OO_{\tilde Q}=\emptyset$. Consider first $(\bar v,\bar g,\bar\vartheta)\in M_2\cup M_4$. Then $\bar v_3=\bar\vartheta=0$ and it follows that $\big((0,0,1),(0,0,0)\big)\in T_{\gph \tilde Q}(\bar v,\bar g,\bar\vartheta)$, but the opposite direction $\big((0,0,-1),(0,0,0)\big)$ cannot belong to the tangent cone because $-1\not\in T_{\R_+}(0)$. Hence, $T_{\gph \tilde Q}(\bar v,\bar g,\bar\vartheta)$ is not a subspace and $(\bar v,\bar g,\bar\vartheta)\not\in\OO_{\tilde Q}$ follows. Finally, let $(\bar v,\bar g,\bar\vartheta)\in M_3^-$. Then for all $t>0$ we have $\big((t\bar g,0),(-\F\bar\vartheta\bar g,\bar \vartheta)\big)\in \gph\tilde Q$ implying $\big((\bar g,0,0),(0,0,0)\big)\in T_{\gph\tilde Q}(\bar v,\bar g,\bar\vartheta)$. Now assume that $-\big((\bar g,0,0),(0,0,0)\big)\in T_{\gph\tilde Q}(\bar v,\bar g,\bar\vartheta)$. By definition, there are sequences $t_k\downarrow 0$ and $(v^k,g^k,\vartheta^k)\longsetto{\gph\tilde Q}(\bar v,\bar g,\bar\vartheta)$ that satisfy $\big((v^k,g^k,\vartheta^k)-(\bar v,\bar g,\bar\vartheta)\big)/t_k\to \big((-\bar g,0),(0,0,0)\big)$. From $(v^k_{12}-\bar v_{12})/t_k=v^k_{12}/t_k\to -\bar g$,  $\vartheta^k\to\bar\vartheta$ and $\norm{\bar g}=-\F\bar\vartheta$ we deduce
\[\lim_{k\to\infty}g^k=-\lim_{k\to\infty}\F\vartheta^k\frac{ v_{12}^k}{\norm{v_{12}^k}}=-\lim_{k\to\infty}\F\vartheta^k\frac{ v_{12}^k/t_k}{\norm{v_{12}^k}/t_k}=-\F\bar\vartheta\frac{-\bar g}{\norm{\bar g}}= -\bar g\]
contradicting $g^k\to\bar g$  and we conclude that $(\bar v,\bar g,\bar\vartheta)\not\in\OO_{\tilde Q}$. This completes the proof.
\end{proof}
Note that in the formulas \eqref{eq-14} and \eqref{eq-15} the matrices $-\mathcal{F}\bar{\vartheta}/\| \bar{v}_{12}\|\left ( I_2- \bar{v}_{12}\bar{v}_{12}^{T}/\| \bar{v}_{12} \|^{2} \right )$ are unbounded for $\bar v_{12}\to 0$. Theoretically, this does not cause problems, because convergence is related to \SCD regularity, which is independent from the basis representation used of the underlying subspaces. However, the use of ill-conditioned bases might produce numerical instability when computing the Newton direction. For this reason, we present another representation of the collections $\Sp\tilde Q(\bar{v},\bar{g},\bar{\vartheta})$ and $\Sp^*\tilde Q(\bar{v},\bar{g},\bar{\vartheta})$ with a well-conditioned base when $(\bar{v},\bar{g},\bar{\vartheta})\in M_1$. Observe that for any two $n\times n$ matrices $A,B$ and every nonsingular $n\times n$ matrix $C$ there are $\rge(A,B)=\rge(AC,BC)$. Thus, the following corollary follows from \eqref{eq-14}, \eqref{eq-15} by using the scaling matrix
\[C=\left(\begin{matrix}
  \left(I_2-\mathcal{F}\bar{\vartheta}\frac{1}{\| \bar{v}_{12}\|}
\left ( I_2- \frac{\bar{v}_{12}\bar{v}_{12}^{T}}{\| \bar{v}_{12} \|^{2}} \right )\right)^{-1}&0\\0&1
\end{matrix}\right)=\left(\begin{matrix}\frac{\norm{\bar v_{12}}}{\norm{\bar v_{12}}-\F\bar\vartheta} I_2- \frac{\F\bar\vartheta}{\norm{\bar v_{12}}-\F\bar\vartheta}\frac{\bar v_{12}\bar v_{12}^T}{\norm{\bar v_{12}}^2}&0\\0&1
\end{matrix}\right).\]
\begin{corollary}
  For $(\bar v,\bar g,\bar\vartheta)\in M_1$ we have
  \begin{align}
   \label{EqM1Alt1} \Sp\tilde Q(\bar v,\bar g,\bar\vartheta)=\left\{\rge\left[
    \left(\begin{matrix}\frac{\norm{\bar v_{12}}}{\norm{\bar v_{12}}-\F\bar\vartheta} I_2- \frac{\F\bar\vartheta}{\norm{\bar v_{12}}-\F\bar\vartheta}\frac{\bar v_{12}\bar v_{12}^T}{\norm{\bar v_{12}}^2}&0\\0&0
\end{matrix}\right), \left(\begin{matrix}\frac{-\F\bar\vartheta}{\norm{\bar v_{12}}-\F\bar\vartheta}\left(I_2-\frac{\bar v_{12}\bar v_{12}^T}{\norm{\bar v_{12}}^2}\right)&
-\F\frac{\bar v_{12}}{\norm{\bar v_{12}}}\\0&1
\end{matrix}\right)    \right]\right\}\\
  \label{EqM1Alt2}  \Sp^*\tilde Q(\bar v,\bar g,\bar\vartheta)=\left\{\rge\left[
    \left(\begin{matrix}\frac{\norm{\bar v_{12}}}{\norm{\bar v_{12}}-\F\bar\vartheta} I_2- \frac{\F\bar\vartheta}{\norm{\bar v_{12}}-\F\bar\vartheta}\frac{\bar v_{12}\bar v_{12}^T}{\norm{\bar v_{12}}^2}&0\\\F\frac{\bar v_{12}^T}{\norm{\bar v_{12}}}&0
\end{matrix}\right), \left(\begin{matrix}\frac{-\F\bar\vartheta}{\norm{\bar v_{12}}-\F\bar\vartheta}\left(I_2-\frac{\bar v_{12}\bar v_{12}^T}{\norm{\bar v_{12}}^2}\right)&
0\\0&1
\end{matrix}\right)    \right]\right\}.
  \end{align}
\end{corollary}

\noindent From Table 1 one can further infer that
\begin{enumerate}
 \item [(i)]
 every point from $M_{2}$ is accessible by sequences belonging solely to $L$ or to $M_{1}$;
 \item [(ii)]
 every point from $M_{3}^{-}$ is accessible by sequences belonging to $M_{1}$ or to $M_{3}^{+}$, and
 \item [(iii)]
 the singleton $M_{4}=\{(0,0,0)\}$ is accessible by sequences belonging to $L$ or to $M_{1}$  or to $M_3^-$ or to $M_{3}^{+}$.
\end{enumerate}
This implies in particular that
\begin{equation}\label{eq-17}
\begin{array}{l}
\rge (I_3,0)\in \Sp^* \tilde Q(\bar{v},\bar{g},\bar{\vartheta})
\mbox{ for } (\bar{v},\bar{g},\bar{\vartheta}) \in M_{2} \cup M_{4}\mbox{ and }\\
 \rge (0,I_3)\in \Sp^* \tilde Q(\bar{v},\bar{g},\bar{\vartheta})
\mbox{ for } (\bar{v},\bar{g},\bar{\vartheta}) \in M_{3}^{-}.
\end{array}
\end{equation}
Formulas (\ref{eq-17}) are used in the implementation of the Newton step of the SCD \ssstar Newton method to the numerical solution of (\ref{eq-8}) in the next section.

Obviously, the mapping $Q^R$ given by \eqref{eq-10} is Lipschitzian and, therefore, graphically Lipschitzian of dimension $n-3p$ as well. Further $\OO_{Q^R}=\gph Q^R=\R^{n-3p}\times\{0\}$ and
\[\Sp Q^R(v,0)=\Sp^* Q^R(v, 0)=\{\rge(I_{n-3p},0)\},\ v\in \R^{n-3p}.\]

Combining Lemma \ref{LemSCDProduct}, Proposition \ref{PropSCDProduct} and Lemma \ref{LemGraphLisch} with Proposition \ref{PropGraphLipTildeQ} and Proposition \ref{prop:SCD} we arrive at the following corollary.

\begin{corollary}
The mapping $Q$ given by \eqref{eq-9} is a \SCD mapping,
\begin{align*}
  \OO_Q=\left\{\big((u^1,\ldots,u^p,u^R), (w^1,\ldots,w^p,0)\big)\mv (u^i,w^i)\in \OO_{\tilde Q}, u^R\in \R^{n-3p}\right\}
\end{align*}
and for every $(u,w)=\big((u^1,\ldots,u^p,u^R),(w^1,\ldots,w^p,0)\big)\in\gph Q$ we have
\begin{align*}
&\lefteqn{\Sp Q(u,w)=}\\
&\left\{\left.\rge\left[\left(\begin{matrix}
  U^1&&&0\\&\ddots\\&&U^p\\0&&&I_{n-3p}\end{matrix}\right), \left(\begin{matrix}
  W^1&&&0\\&\ddots\\&&W^p\\0&&&0\end{matrix}\right)
  \right]\,\right|\,\begin{array}{l}\rge(U^i,W^i)\in\Sp\tilde Q(u^i,w^i),\\ i=1,\ldots,p\end{array}\right\}\\
&\lefteqn{\Sp^* Q(u,w)=}\\
&\left\{\left.\rge\left[\left(\begin{matrix}
  {W^*}^1&&&0\\&\ddots\\&&{W^*}^p\\0&&&I_{n-3p}\end{matrix}\right), \left(\begin{matrix}
  {U^*}^1&&&0\\&\ddots\\&&{U^*}^p\\0&&&0\end{matrix}\right)
  \right]\,\right|\,\begin{array}{l}\rge({W^*}^i,{U^*}^i)\in\Sp^*\tilde Q(u^i,w^i),\\ i=1,\ldots,p\end{array}\right\}.
\end{align*}
\end{corollary}

To implement the \ssstar Newton method, we must also specify the approximation step. For every $\gamma>0$ the mapping $\big(\gamma I_{n-3p}+Q^R\big)^{-1}=I_{n-3p}/\gamma$ is obviously single-valued and Lipschitzian on $\R^{n-3p}$. Together with Proposition \ref{PropGraphLipTildeQ} we obtain that
\[(\gamma I_n+Q)^{-1}=\myvec{(\gamma I_3+\tilde Q)^{-1}\\\vdots\\(\gamma I_3+\tilde Q)^{-1}\\(\gamma I_{n-3p}+Q^R)^{-1}}\]
and consequently also the mapping $(I_n+\frac 1\gamma Q)^{-1}=(\gamma I_n+Q)^{-1}\circ \gamma I_n$ has these properties. Thus, for a given iterate $\ee uk$, the choice
\begin{equation}\label{EqApprStepCoulomb}
  \ee{\hat d}k:=(I_n+\frac 1\gamma Q)^{-1}\big( \ee uk-\frac 1\gamma (A\ee uk-l)\big)=(\gamma I_n+ Q)^{-1}\big( \gamma \ee uk-(A\ee uk-l)\big)
\end{equation}
is feasible for the approximation step by Proposition \ref{PropApprStepResolvent}.

\section{Numerical experiments}

We treat various geometries arising from the cuboid $(0,2)\times(0,1)\times(0,1) \subset\R^3$ by modifying its bottom surface. Given a function $d:(0,2)\times(0,1)\mapsto\R$, we consider the elastic body represented by the domain
\[\Omega(d):=\{(x_1,x_2,x_3)\mv (x_1,x_2)\in(0,2)\times(0,1),\ d(x_1,x_2)<x_3<1\}.\]
At the left surface $x_1=0$ of the body, we impose Dirichlet boundary conditions and on the top surface $x_3=1$ and the right surface $x_1=2$ act surface tractions with densities $P_T$ and $P_R$. The rigid obstacle is given by the half-space $\R^2\times\R_-$ so that the contact boundary is the bottom surface of the body.

Depending on the discretization parameter $lev$, the elastic body is uniformly cut into $n_{x_1}\cdot n_{x_2}\cdot n_{x_3}$ hexahedrons, where
\[n_{x_1}=\lceil 4\cdot2^{lev/2}\rceil,\ n_{x_2}=n_{x_3}=\lceil 2\cdot2^{lev/2}\rceil.\]
This results in a hexahedral mesh with $(n_{x_1}+1)\cdot (n_{x_2}+1)\cdot (n_{x_3}+1)$ vertices, where $(n_{x_2}+1)\cdot(n_{x_3}+1)$ vertices are in the Dirichlet part of the boundary and $p:=n_{x_1}\cdot (n_{x_2}+1)$ vertices belong to the contact part of the boundary. The total number of degrees of freedom of the nodal displacements is $$n:=3n_{x_1}\cdot(n_{x_2}+1)\cdot(n_{x_3}+1).$$ The setting is shown in Figure \ref{fig:setup}.

\begin{figure}
\centering
 \begin{minipage}[c]{.32\textwidth}
\centering
\includegraphics[width=0.99\textwidth]{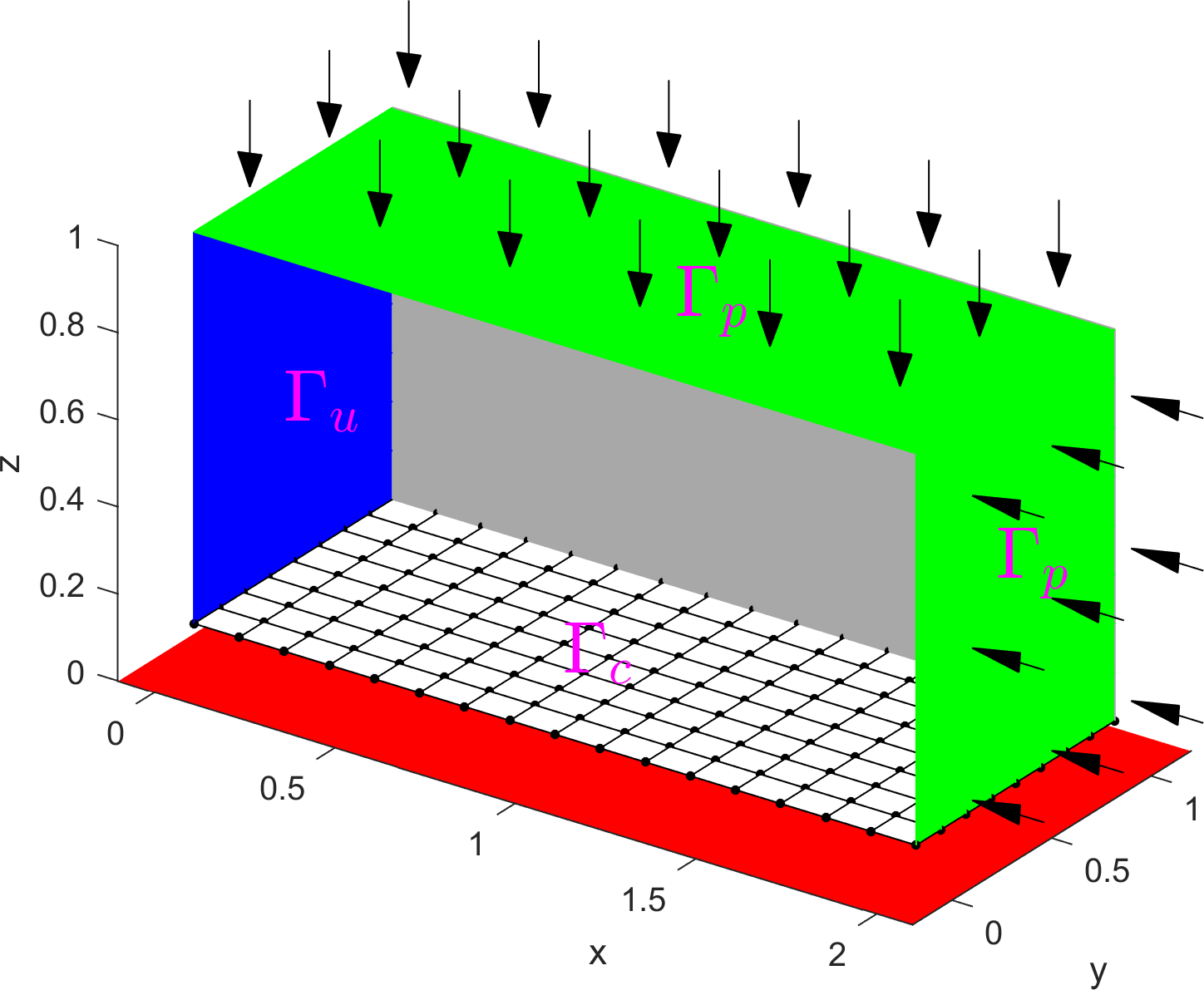}
\end{minipage}
\hspace{.01\textwidth}
\begin{minipage}[c]{.32\textwidth}
\vspace{0.5cm}
\centering
\includegraphics[width=0.99\textwidth]{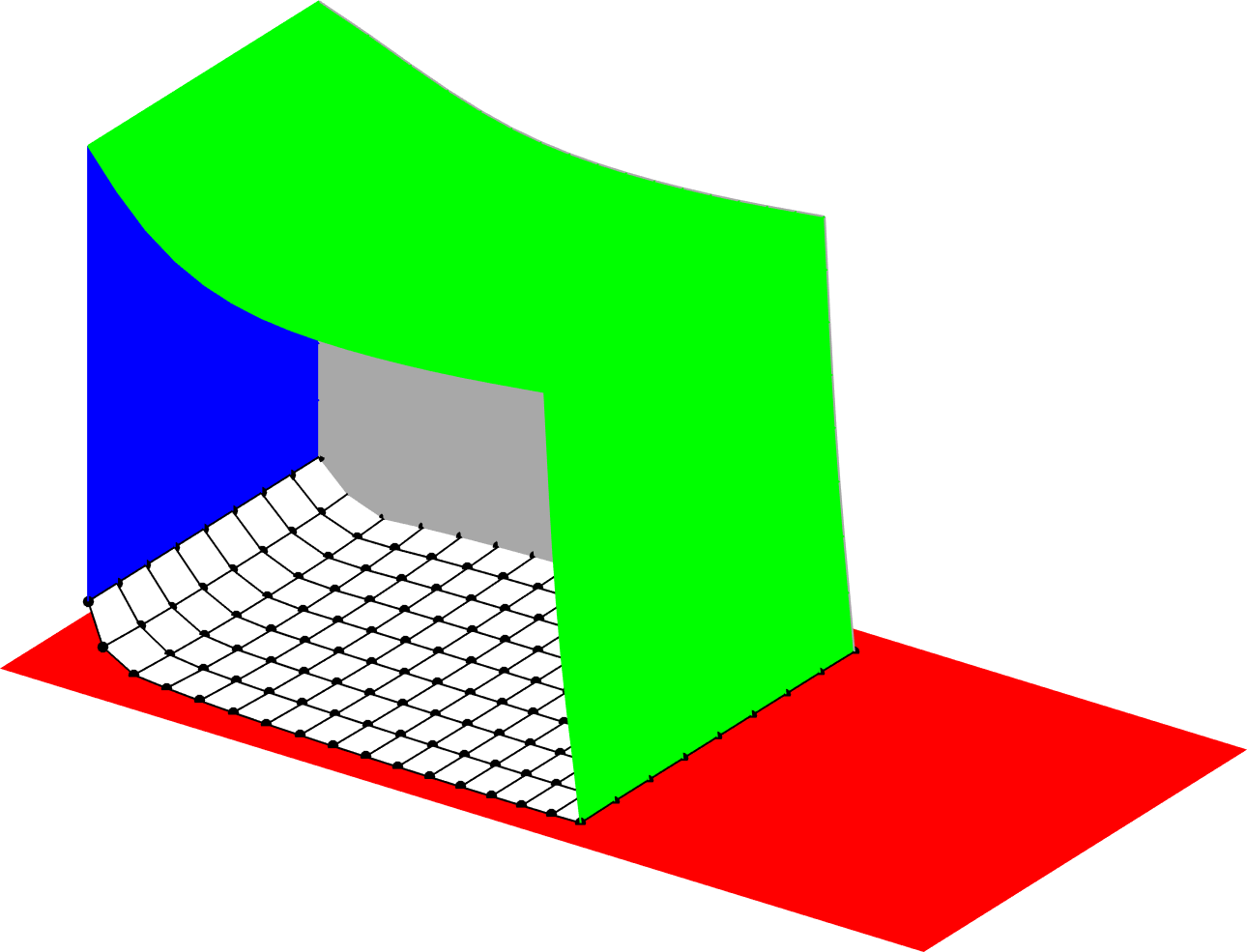}
\end{minipage}
%\hspace{.02\textwidth}
\begin{minipage}[c]{.32\textwidth}
\vspace{0.7cm}
\centering
\includegraphics[width=0.99\textwidth]{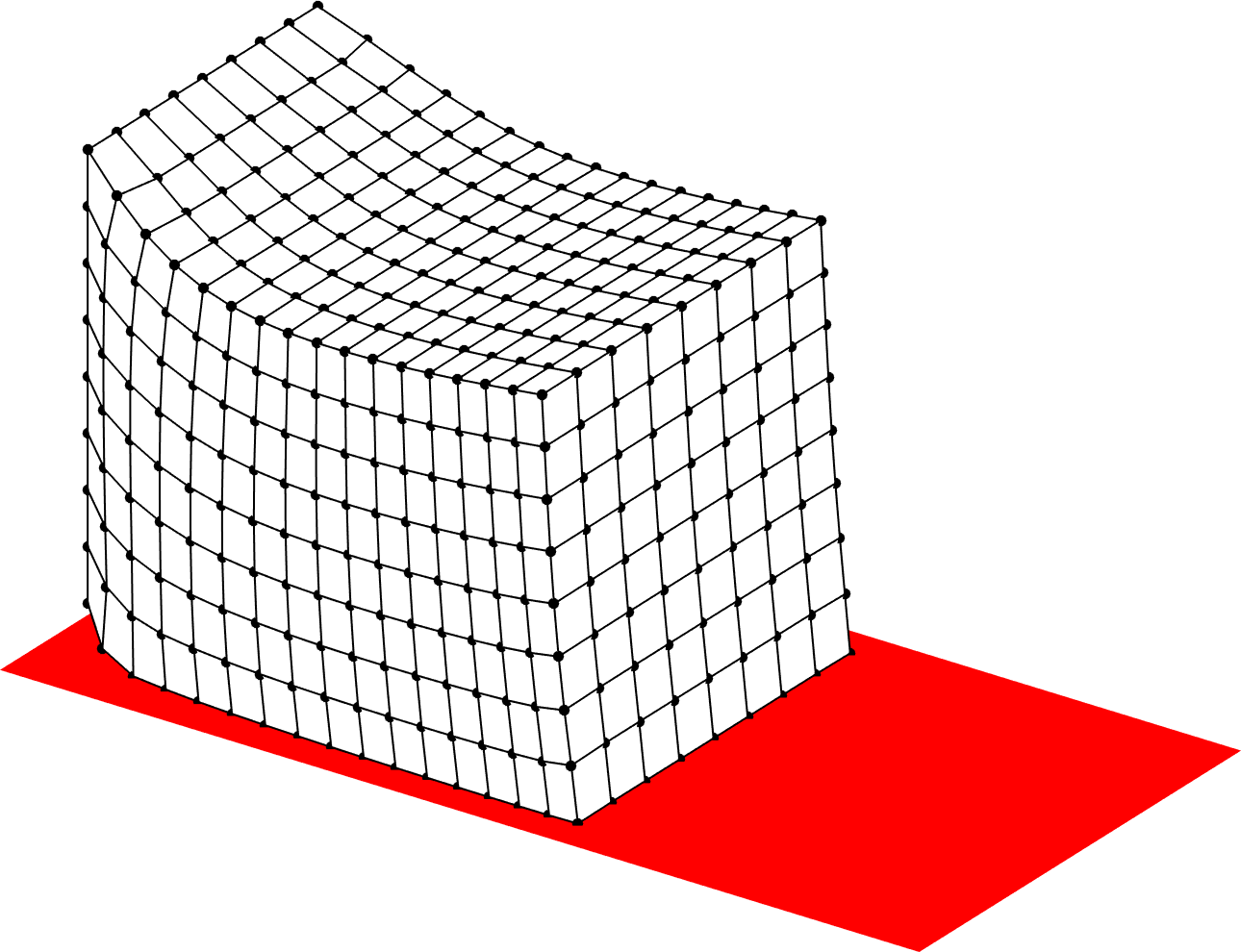}
\end{minipage}
\caption{An undeformed and deformed cuboid domain are shown in the left and middle pictures. The zero Dirichlet boundary condition for displacements is assumed on the blue part of the boundary $\Gamma_u$, surface tractions are applied to the green part of the boundary $\Gamma_p$ and the contact boundary $\Gamma_c$ is pressed against the (red) rigid plane foundation. Front faces are not visualized. The right picture shows the deformed cuboid domain decomposed in hexahedrons.
} \label{fig:setup}
\end{figure}

The resulting GE \eqref{eq-8} is solved with the SCD \ssstar Newton method described in Sections 5 and 6 and the implementation was carried out in MATLAB on a PC with an i7-7700 CPU.
The part of the code that describes the model is built on the original code of \cite{BeHaKoKuOut09} with the accelerated assembly of the elastic stiffness matrix as described in \cite{CSV}. This part of the code was also applied to the Tresca friction solver of \cite{GfrOutValsozopol}. The main difference between the implementations is that in the new code, no reduction is done to the nodes on the contact boundary, and the complete problem is treated with all domain nodes.

For the scalar $\gamma$ used in the approximation step \eqref{EqApprStepCoulomb} we used an approximation of the largest eigenvalue of $A$ obtained by five iterations of the power method. The system \eqref{EqNewtonDir} that defines the Newton direction was solved iteratively using the GMRES method with ILU factorization as a preconditioner. We stopped the GMRES method when the relative residual (non-preconditioned) is less than the prescribed tolerance $tol$. Clearly, in this case, we will lose the superlinear convergence, but we can expect linear convergence with the rate $tol$. Of course, we can use more advanced methods to solve the linear system that determines the Newton direction, but the main task of this paper is to demonstrate the efficiency of the \ssstar Newton method and not the solution of linear systems.

To improve the global convergence properties, we use a non-monotone line search heuristic as introduced in \cite{GfrOutVal21}. Recall that the quantity $\ee{\hat y}k$ given by \eqref{EqResGen} acts as a residual for GE \eqref{EqGEMod} at the point $(\ee xk, \ee{\hat d}k)$. In the context of our contact problem with Coulomb friction, given an iterate $\ee uk$ of nodal displacements and $\ee{\hat d}k$ by \eqref{EqApprStepCoulomb}, we consider
\[\ee{\hat y}k:=(\gamma (\ee uk-\ee{\hat d}k),\ee uk-\ee{\hat d}k)\]
as a residual. If the Newton direction used is denoted by $\Delta \ee uk$, the next iterate is calculated as $\ee u{k+1}=\ee uk+\ee\alpha k\Delta \ee uk$, where $\ee \alpha k$ is chosen as the first element of the sequence $1,\frac 12,\frac 14,\frac 18,\frac 1{32},\frac 1{128}, \frac{0.1}{128},\frac{0.01}{128},\ldots$ such that
\[\norm{\ee{\hat y}{k+1}}\leq\left(1-0.1\ee\alpha k+\frac{0.1}{k+1}\right)\norm{\ee{\hat y}k}.\]
We considered three different elastic bodies $\Omega(d_i)$, $i=1,2,3$, given by
\begin{align*}
  &d_1(x_1,x_2)=0.01,\\
  &d_2(x_1,x_2)=\max\{0.01-0.015\sqrt{0.5(x_1-1)^2+2(x_2-0.5)^2},0.0025\},\\
  &d_3(x_1,x_2)=0.01+0.005\big(\sin(2\pi x_1)+\cos(2\pi x_2)\big)
\end{align*}
and two load cases $L_1,L_2$ with surface tractions
\begin{align*}
  &L_1:\ P_T=(0,0,-1\,{\rm GPa}),\ P_R=(-0.2\,{\rm GPa},0,0),\\
  &L_2:\ P_T=(0,0,-1\,{\rm GPa}),\ P_R=(-0.17\,{\rm GPa},-0.1\,{\rm GPa},0).
\end{align*}
As material parameters, we chose Young's modulus $E:=70\,{\rm GPa}$ and Poisson's ratio $\nu=0.334$ (aluminum). The friction coefficient was always chosen as $\F=0.23$. In Table \ref{TabPerf1} we report for different discretization levels $lev$ the number $p$ of nodes in the contact part of the boundary and the number $n$ of degrees of freedom. Furthermore, using the relative accuracy $tol=0.1$ to calculate Newton's direction, for each of the six possible combinations of geometries $d_1,d_2,d_3$ and load cases $L_1,L_2$ we list the number of Newton iterations $it$ and the total number $gmres$ of GMRES iterations needed to reduce the initial residual $\norm{\ee{\hat y}0}$ by a factor of $10^{-12}$. The starting point $\ee u0$ for the \ssstar Newton method was always chosen as the origin. The value $gmres$ characterizes the computational complexity.

\begin{table}[t]
\begin{tabular}{|c|c|c|c|c|c|c|c|c|}
\hline
\multicolumn{3}{|c|}{}&$d_1/L_1$&$d_1/L_2$&$d_2/L_1$&$d_2/L_2$&$d_3/L_1$&$d_3/L_2$\\
\hline
$lev$&$p$&$n$&$it/gmres$&$it/gmres$&$it/gmres$&$it/gmres$&$it/gmres$&$it/gmres$\\
\hline
3&84&1\,764&13\,/\,774&13\,/\,833&13\,/\,830&13\,/\,833&14\,/\,781&13\,/\,780\\
4&144&3\,888&13\,/\,866&15\,/\,982&15\,/\,868&14\,/\,937&14\,/\,874&14\,/\,882\\
5&299&11\,661&15\,/\,952&15\,/\,1012&16\,/\,986&13\,/\,995&14\,/\,979&15\,/\,919\\
6&544&27\,744&16\,/\,1148&16\,/\,1216&14\,/\,1065&15\,/\,1101&17\,/\,1085&16\,/\,1145\\
7&1\,104&79\,488&15\,/\,1157&17\,/\,1210&14\,/\,1078&15\,/\,1189&15\,/\,1154&16\,/\,1186\\
8&2\,112&209\,088&16\,/\,1402&16\,/\,1332&16\,/\,1301&16\,/\,1443&17\,/\,1437&19\,/\,1538\\
9&4\,277&603\,057&19\,/\,1926&18\,/\,1589&16\,/\,1401&17\,/\,1692&19\,/\,1722&18\,/\,1714\\
10&8\,320&1\,622\,400&19\,/\,1864&17\,/\,1768&18\,/\,1896&19\,/\,1880&19\,/\,1920&19\,/\,2122\\
\hline
\end{tabular}
\caption
 {\label{TabPerf1}Iteration numbers for  $tol=0.1$ and a starting point $\ee u0=0$.}
\end{table}

We can see that for every geometry and every load case the iteration numbers are nearly equal and increase only slightly with finer discretizations.

In Figure \ref{FigContact} we depict for the different cases the undeformed bottom surface and the contact states for the deformed contact boundary.

\begin{figure}[t]
    \includegraphics[width=0.27\textwidth]{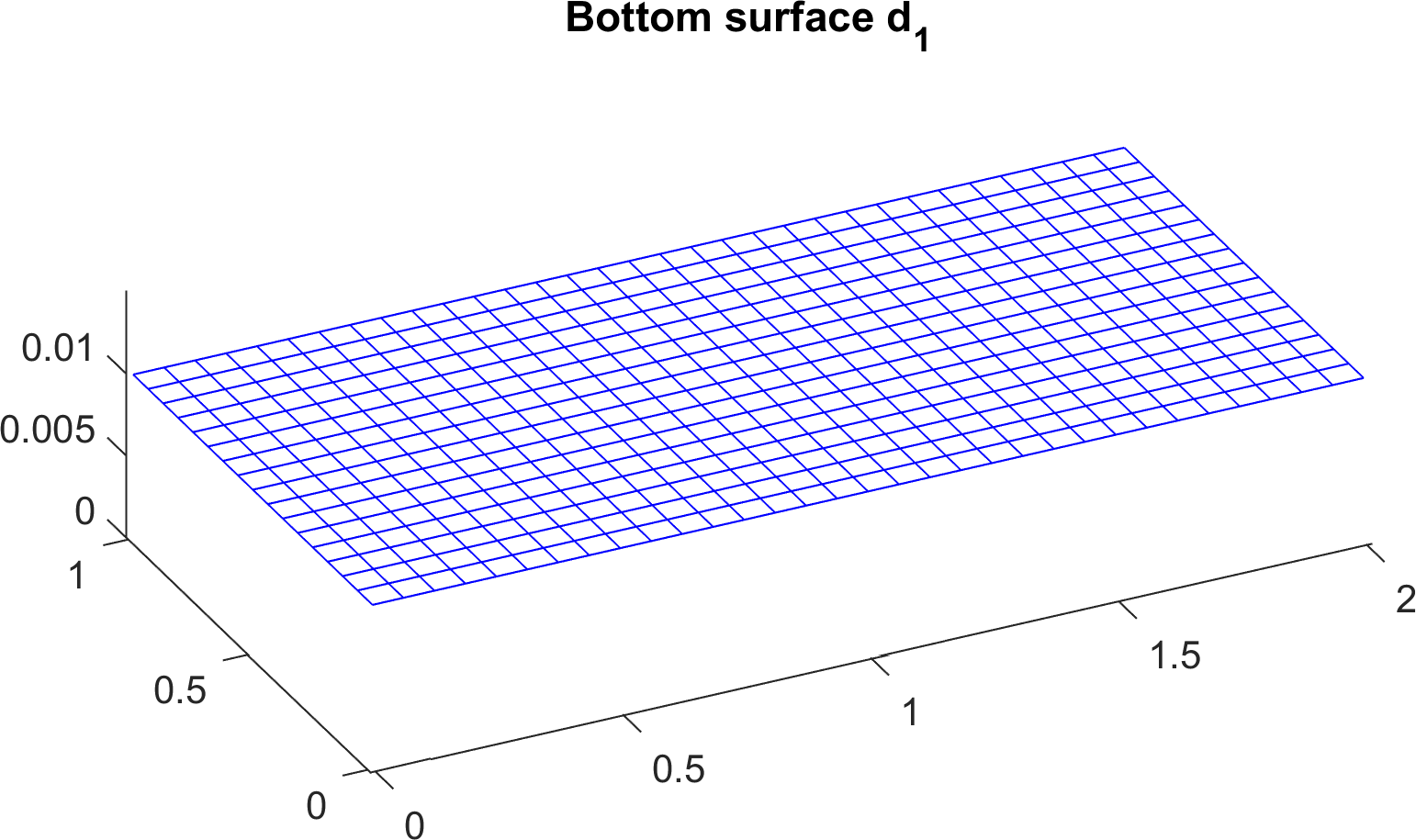}\
    \includegraphics[width=0.36\textwidth]{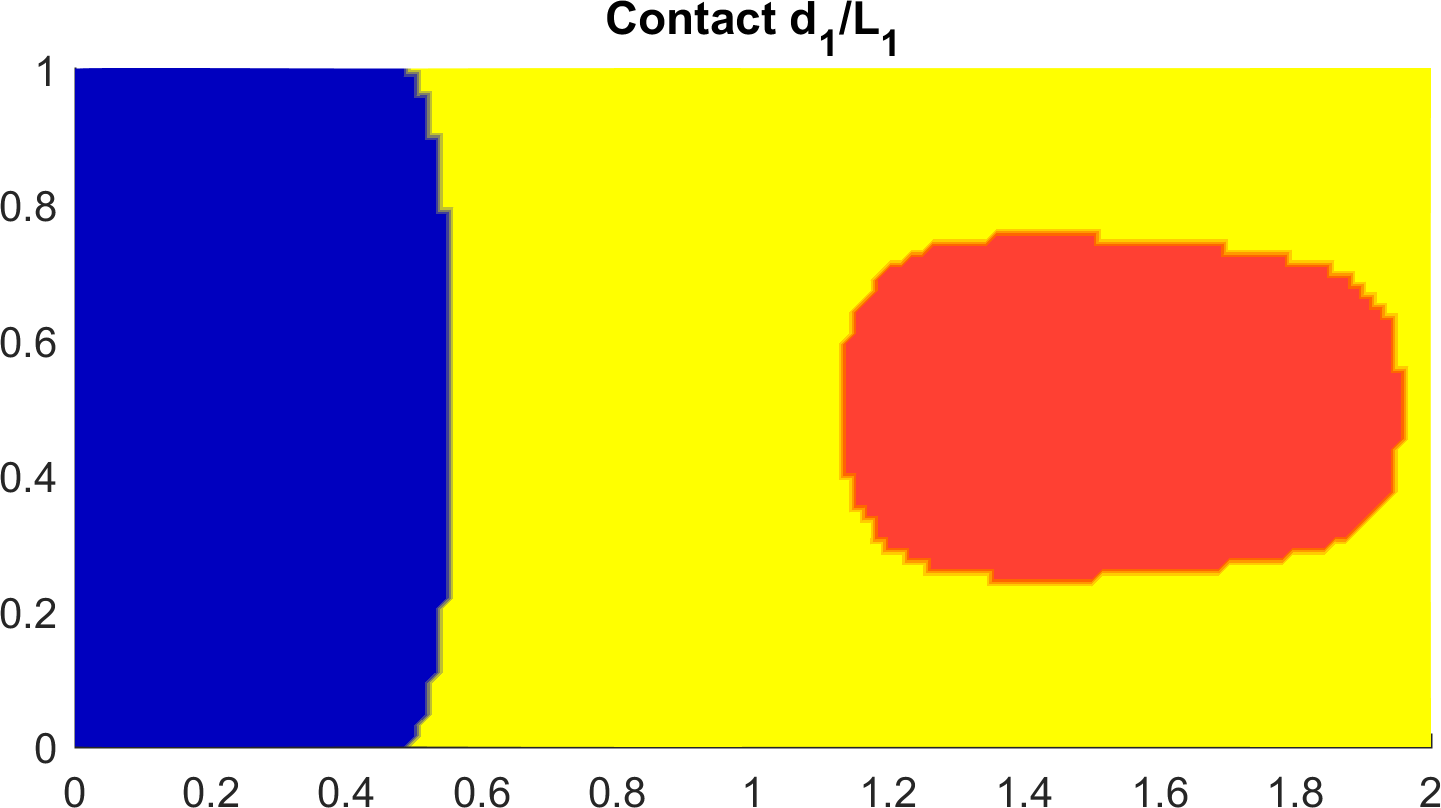}\
    \includegraphics[width=0.36\textwidth]{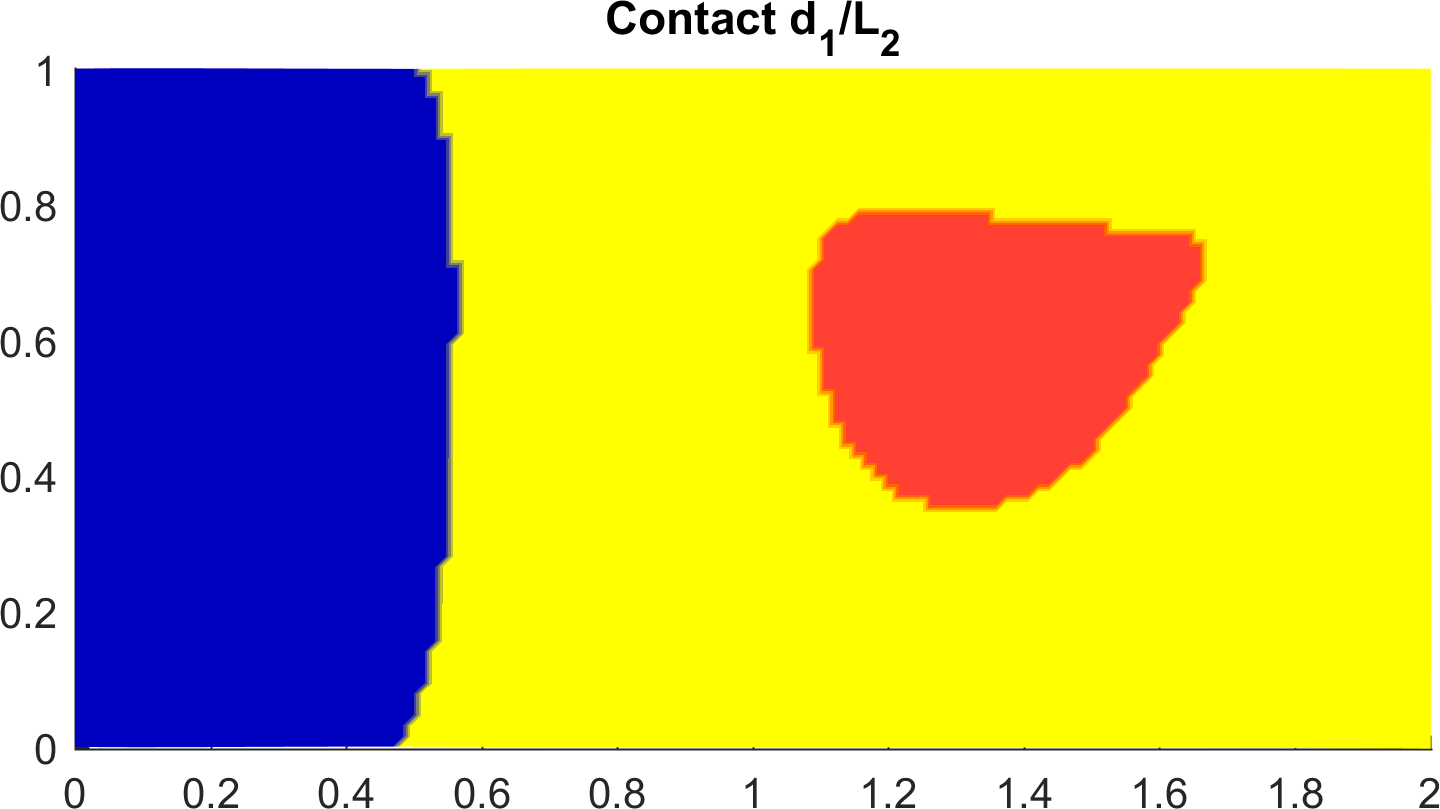}\\[1ex]
    \includegraphics[width=0.27\textwidth]{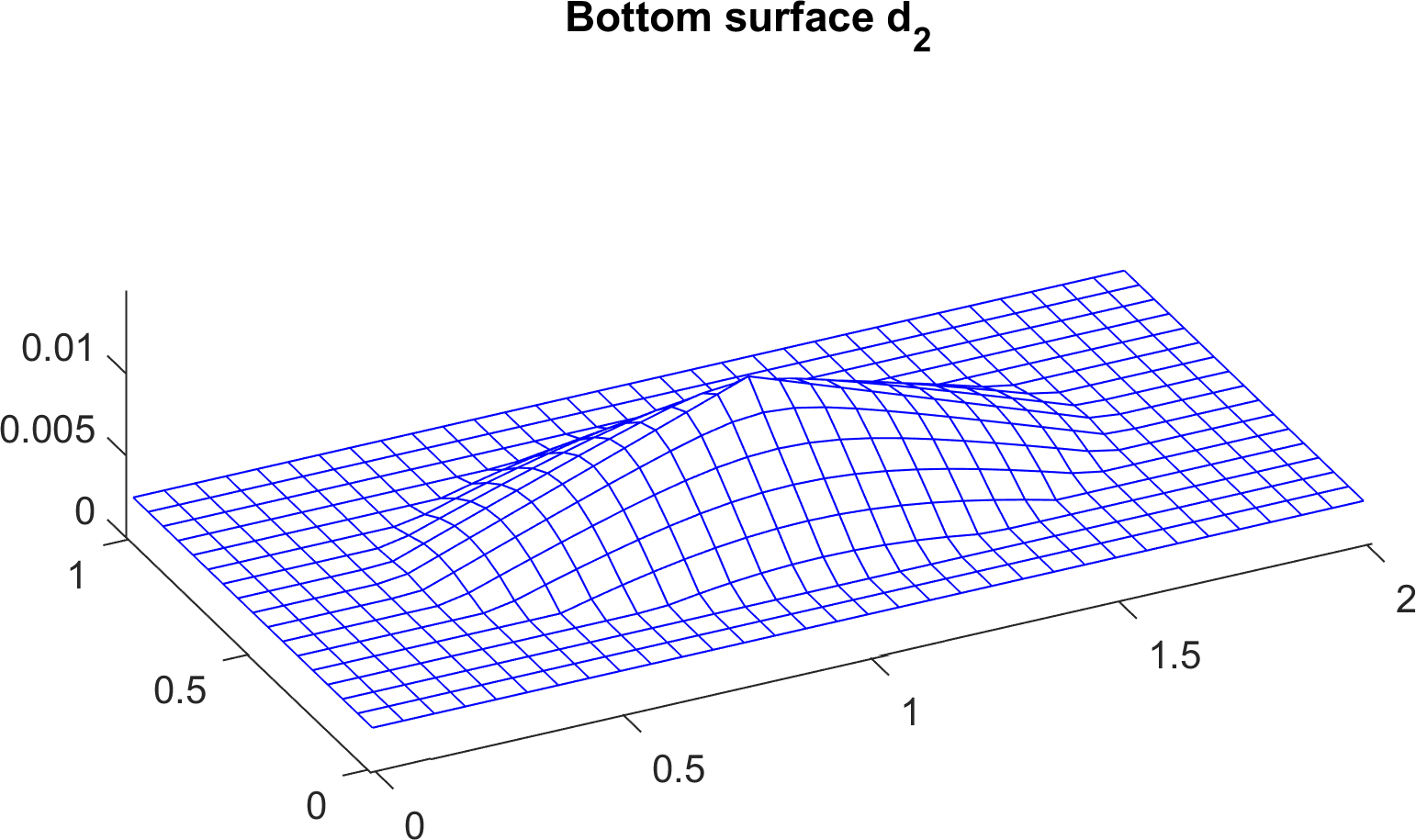}\
    \includegraphics[width=0.36\textwidth]{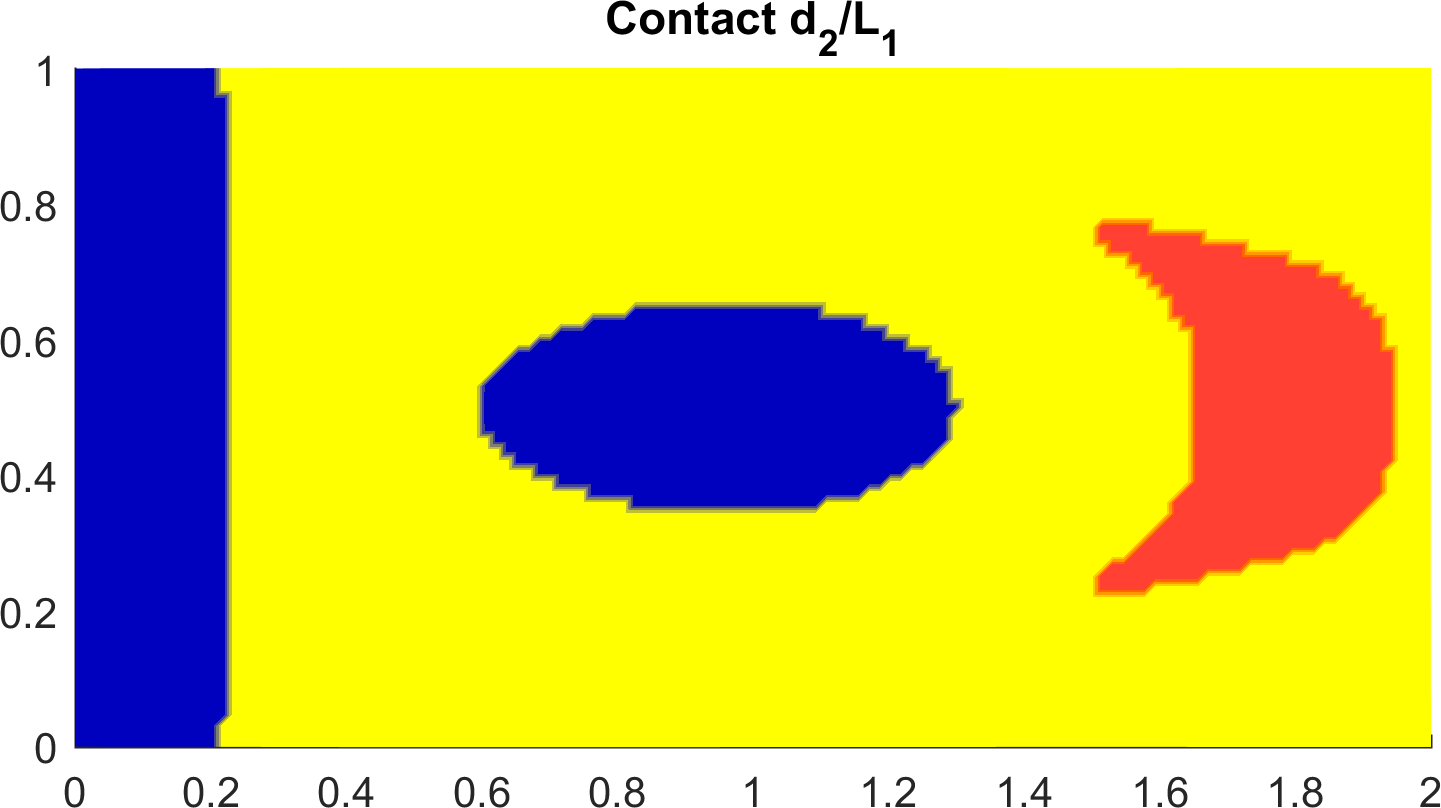}\
    \includegraphics[width=0.36\textwidth]{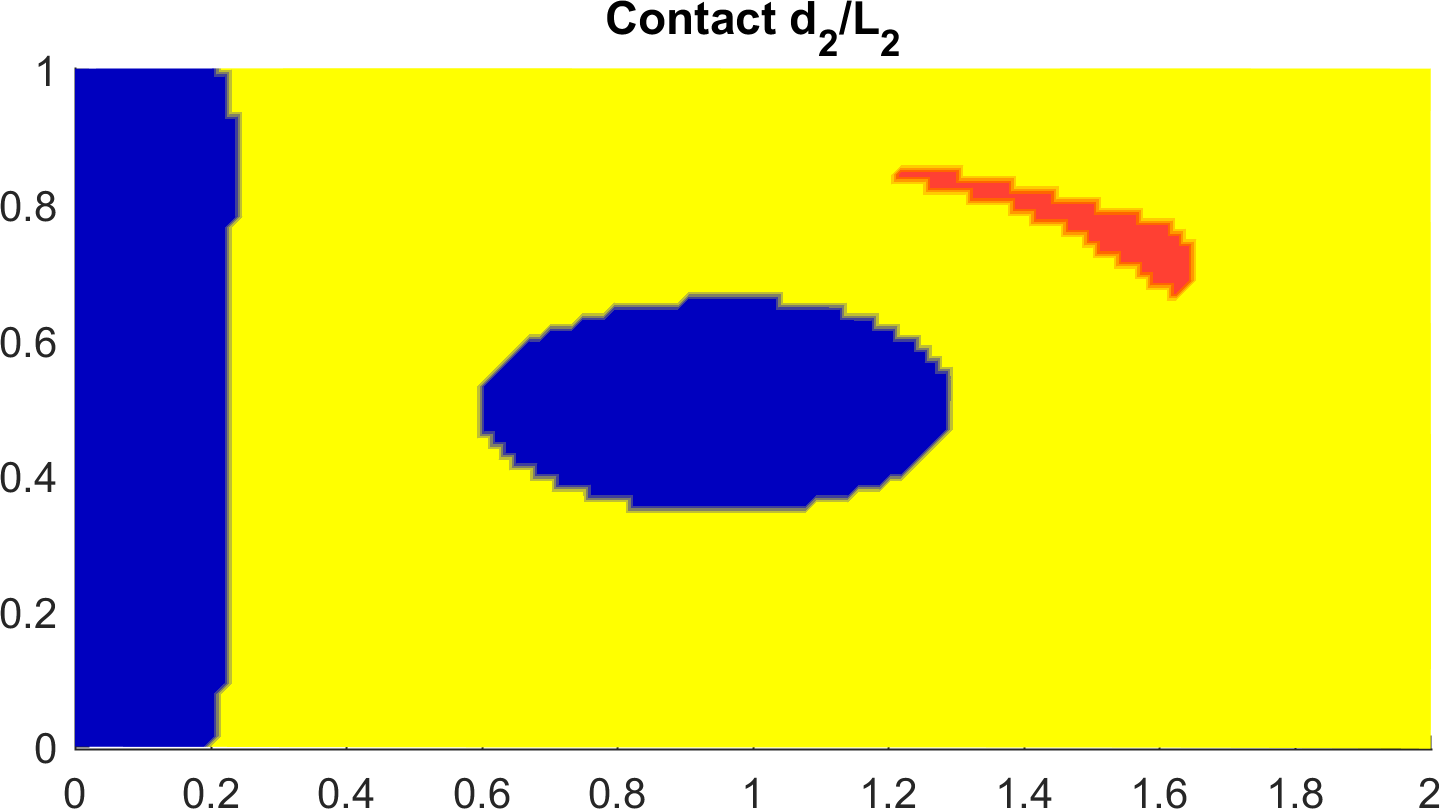}\\[1ex]
    \includegraphics[width=0.27\textwidth]{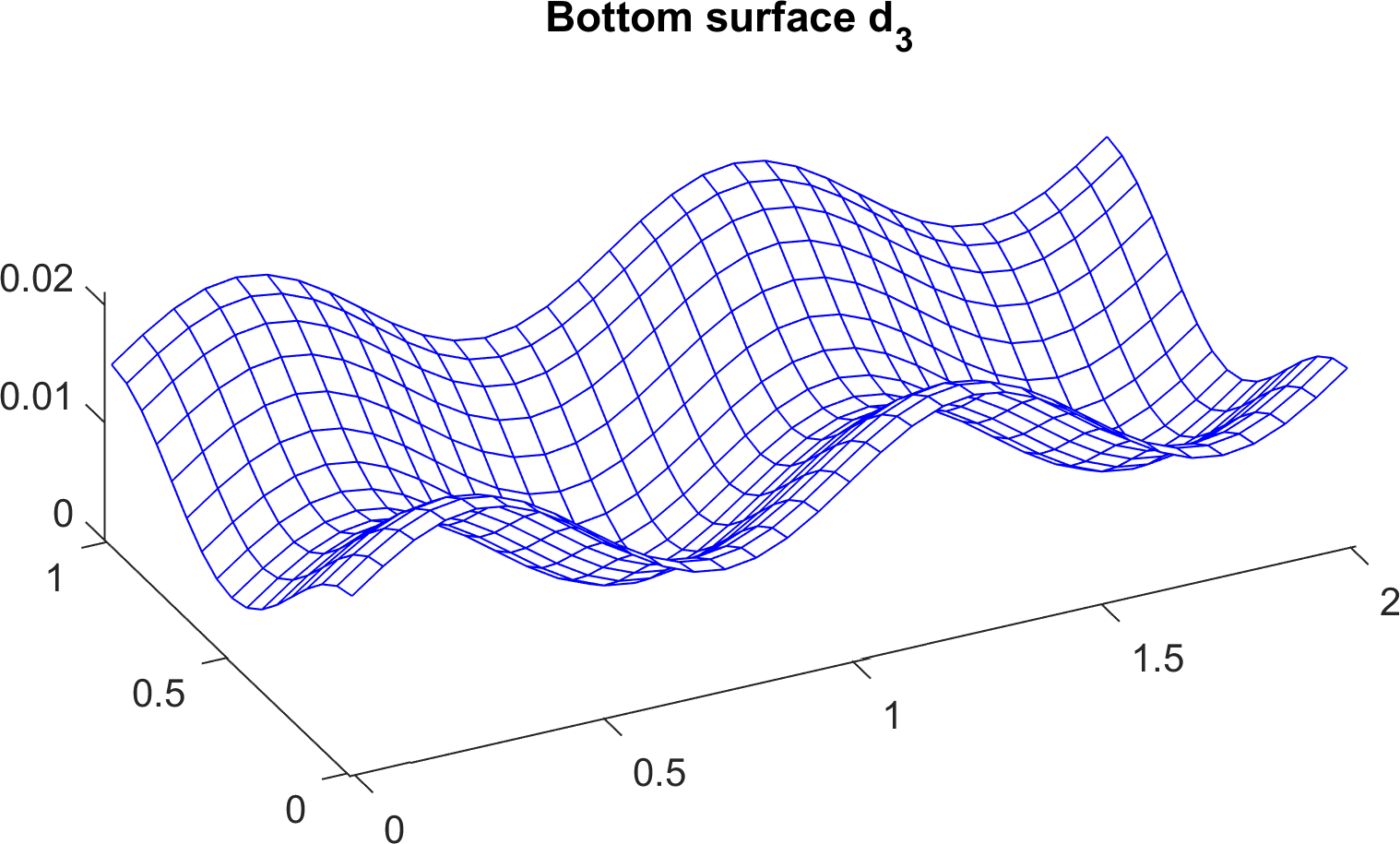}\
    \includegraphics[width=0.36\textwidth]{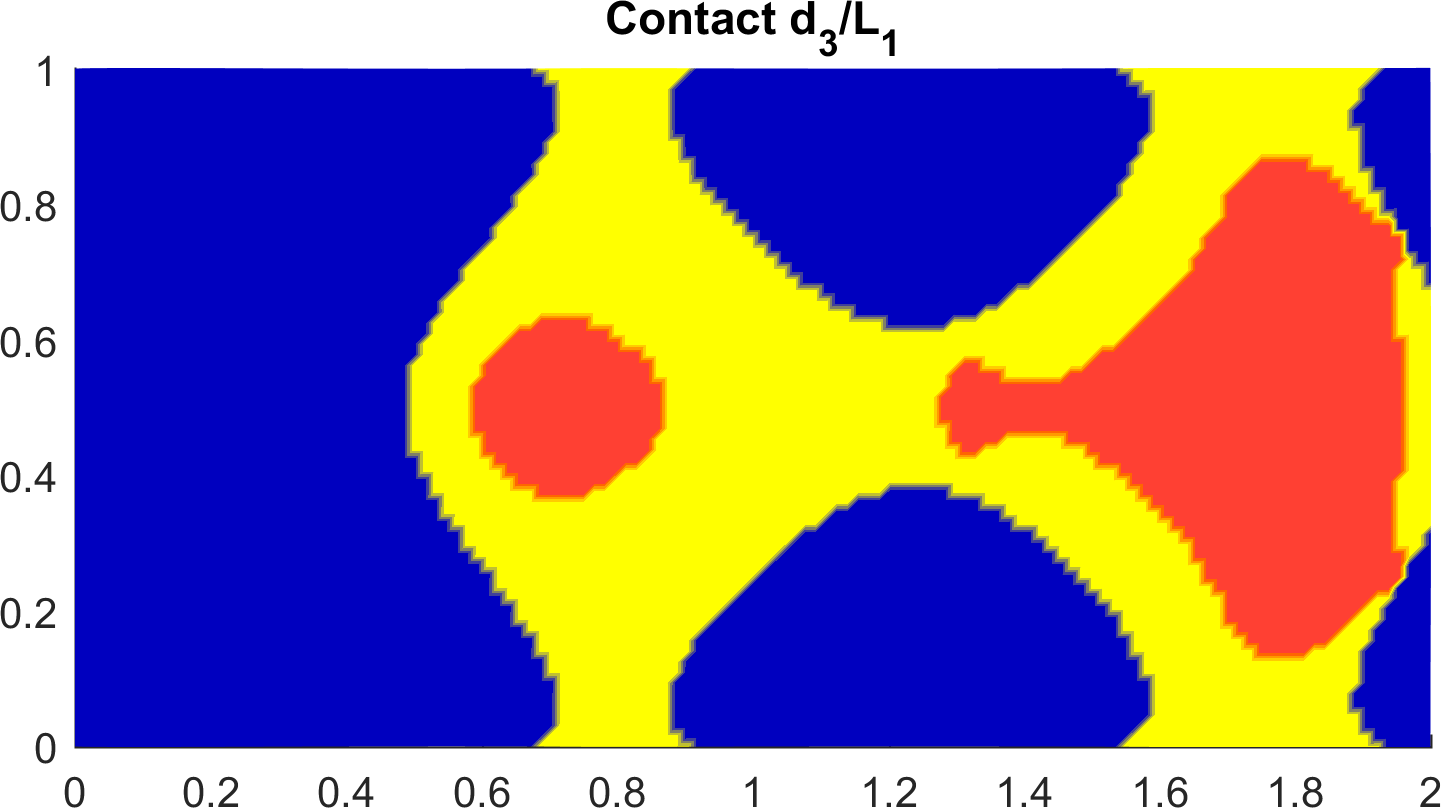}\
    \includegraphics[width=0.36\textwidth]{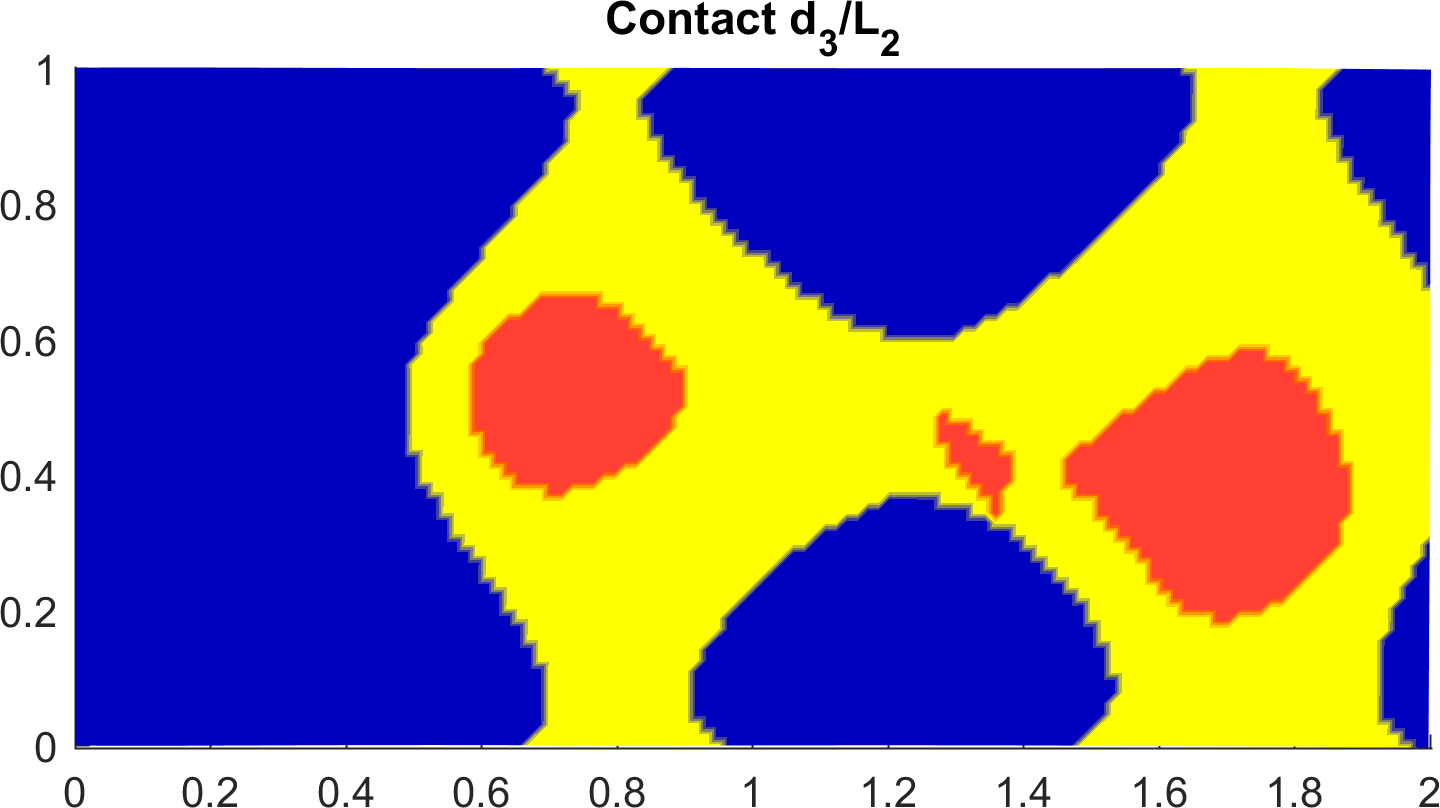}
    \caption{\label{FigContact}Undeformed bottom surface and  states in the deformed contact boundary: no contact (blue), sliding (yellow), sticking (red).}
\end{figure}

The number of iterations will decrease when a better starting point is available. In Table \ref{TabPerf2} we display the iteration numbers when we use as a starting point the solution of the previous discretization level interpolated to the finer mesh.

\begin{table}[t]
\centering
\begin{tabular}{|c|c|c|c|c|c|c|}
\hline
&$d_1/L_1$&$d_1/L_2$&$d_2/L_1$&$d_2/L_2$&$d_3/L_1$&$d_3/L_2$\\
\hline
$lev$&$it/gmres$&$it/gmres$&$it/gmres$&$it/gmres$&$it/gmres$&$it/gmres$\\
\hline
4&11\,/\,678&11\,/\,678&12\,/\,624&11\,/\,629&11\,/\,620&10\,/\,622\\
5&11\,/\,703&11\,/\,652&11\,/\,643&12\,/\,717&11\,/\,650&11\,/\,781\\
6&13\,/\,778&11\,/\,744&11\,/\,719&12\,/\,822&11\,/\,802&11\,/\,767\\
7&11\,/\,781&9\,/\,717&11\,/\,786&11\,/\,849&11\,/\,862&10\,/\,905\\
8&11\,/\,763&12\,/\,786&12\,/\,858&11\,/\,937&11\,/\,842&11\,/\,981\\
9&10\,/\,867&11\,/\,997&11\,/\,1043&13\,/\,1179&11\,/\,1052&10\,/\,1058\\
10&11\,/\,1087&10\,/\,958&11\,/\,1147&12\,/\,1050&11\,/\,1092&12\,/\,1123\\\hline
\end{tabular}
\caption
 {\label{TabPerf2}Iteration numbers for  $tol=0.1$ and a starting point $\ee u0$ set as the interpolated solution of the previous discretization level.}
\end{table}

We observe that the number of GMRES iterations is reduced by $35-45\%$ at the highest discretization level. A closer analysis shows that, as expected, we essentially avoid iterations to localize the solution. In fact, after 1--3 iterations, we have identified the correct state of all nodes in the contact part of the boundary, and the remaining iterations are only needed to reach the desired accuracy. Note that we use $tol=0.1$ and therefore expect linear convergence with the rate $0.1$. Since we want to reduce the residual by a factor of $10^{-12}$, we expect about 12 \ssstar Newton steps to achieve this goal. In most cases, we need fewer iterations: The reason is that the relative residual of the calculated direction is sometimes significantly less than $tol$.

Finally, we investigate the impact of the parameter $tol$ on the performance of the \ssstar Newton method. Here, we consider only the load case $L_2$ and that the bottom surface is given by $d_3$ with the discretization level $lev=10$. In Table \ref{TabPerf3} we report the iteration numbers $it$ of the \ssstar Newton method and the total number $gmres$ of GMRES iterations for the starting point $\ee u0=0$. We can see that for $tol=0.1$ we need the most \ssstar Newton steps; however, the total number of GMRES iterations, which measures computational complexity, is the lowest.

\begin{table}[t]
\centering
\begin{tabular}{|c|c|c|c|c|}
\hline
$tol$&$10^{-1}$&$10^{-2}$&$10^{-3}$&$10^{-4}$\\
\hline
$it/gmres$&19\,/\,2122&14\,/\,4438&12\,/\,5349&12/7337\\
\hline
\end{tabular}
\caption{\label{TabPerf3}Iteration numbers for various values of $tol$ (case $d_3/L_2$, $lev=10$, $\ee u0=0$).}
\end{table}

We show the convergence of the \ssstar Newton method for the four values of $tol$ in Figure \ref{FigConv}. We see that during the first 5 or 6 iterations, when the \ssstar Newton method tries to localize the solution, the accuracy $tol$ does not play any role in decreasing the residual, and we only need a lot of $GMRES$ iterations to compute the search directions with higher accuracy. As soon as we are sufficiently close to the solution, the increased accuracy for computing the search direction also yields better convergence rates and consequently fewer iterations for the \ssstar Newton method. However, we also need more GMRES iterations to calculate the search direction, which defeats the advantage of a better convergence rate.

\begin{figure}[t]
\centering
    \includegraphics[width=0.5\textwidth]{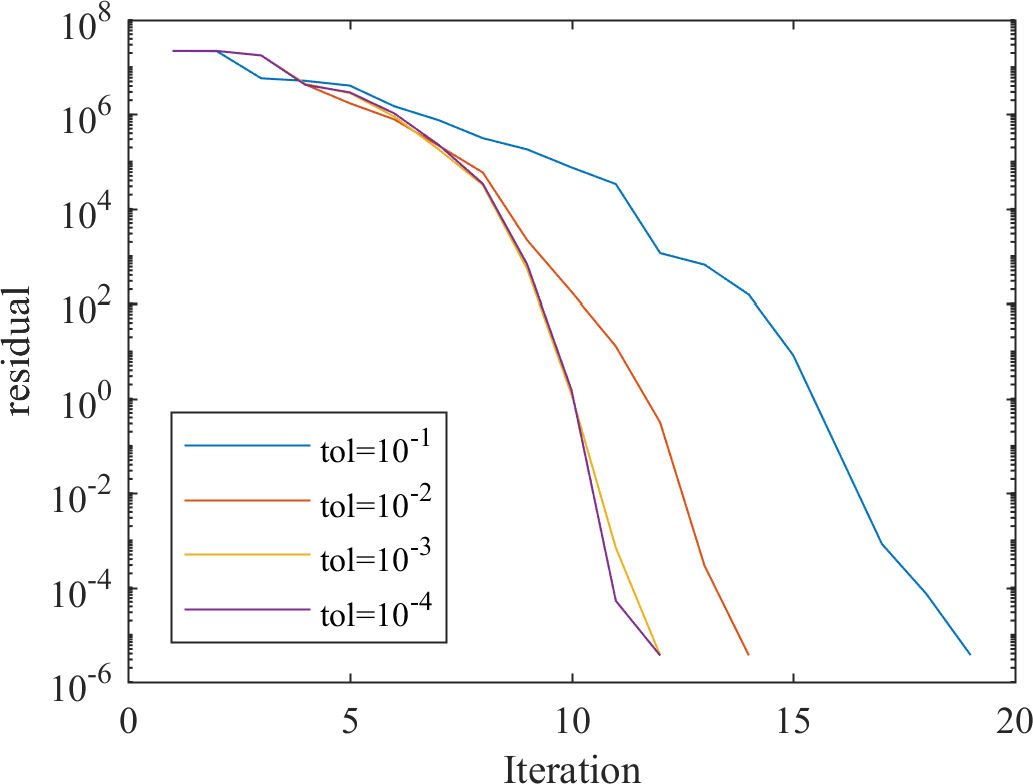}
    \caption{\label{FigConv}Convergence of the \ssstar Newton method for different values of $tol$ (Case $d_3/L_2$, $lev=10$, $\ee u0=0$).}
     \end{figure}

   \section{Conclusion}
The paper shows the abilities of the SCD \ssstar Newton method to compute, apart from the variational inequalities of the first and second kind, also a more complicated class of equilibria which can be modeled as GEs with an SCD and \ssstar multi-valued part. This is documented by a large-scale highly complicated contact problem, where the efficiency of the method enables us, in contrast to most existing approaches, to solve the respective GE on the whole domain without the time-consuming reduction to nodes lying on the contact boundary. We do hope that the SCD \ssstar Newton method will exhibit a similar performance also in some other mechanical problems having a similar structure as the considered contact problem with Coulomb friction.

The paper is dedicated to our friend A.L. Dontchev, for whom nonsmooth Newton methods definitely belonged to favorite research topics and who contributed to the development of this area in a remarkable way, cf., e.g., \cite{DonRo}, \cite{CiDoK}.

\section*{ Acknowledgement}
The authors are indebted to J. Haslinger and P. Beremlijski for valuable advices concerning the modeling and discretization of static contact problems with Coulomb friction. Further, our sincere thanks are due to both reviewers for their valuable suggestions.

\section{Appendix}
%\noindent{\large \bf Appendix}\\
The next statement completes the assertion of Proposition \ref{prop:SCD}.\\

\begin{proposition}
Consider the mapping $\widetilde{Q}$ and assume that
 $(\bar{v},\bar{g},\bar{\vartheta})\in M_{2}$. Then
\[ \Sp\widetilde{Q}(\bar{v},\bar{g},\bar{\vartheta})=
 \left \{ \rge(I_3,0),\
 \rge  \left [  \left (\begin{matrix}
I_{2} & 0\\
0 & 0
\end{matrix}\right),
\left(\begin{matrix}
0 &
 -\mathcal{F}  \frac{\bar{v}_{12}}{\| \bar{v}_{12} \|} \\
0 & 1
\end{matrix}\right )
\right ]\right \}.
\]
For $(\bar{v},\bar{g},\bar{\vartheta}) \in M_{3}^{-}$ one has
\begin{equation}\label{eq-32}
\Sp \widetilde{Q}(\bar{v},\bar{g},\bar{\vartheta})=
 \left \{ \rge(0,I_3),\
 \rge  \left[
 \left(\begin{matrix}
\frac{\bar g\bar g^T}{\norm{\bar g}^2}& 0\\
0 & 0
\end{matrix}\right ),
\left (\begin{matrix}
I_2-\frac{\bar g\bar g^T}{\norm{\bar g}^2} &
 -\mathcal{F} \frac{\bar{g}}{\norm{\bar g}}\\
0 & 1
\end{matrix}\right )
  \right ]\right \}
\end{equation}
and for $(\bar{v},\bar{g},\bar{\vartheta})=(0,0,0)\in M_{4}$ we have
\begin{align}
 \nonumber&\lefteqn{\Sp\widetilde{Q}(\bar{v},\bar{g},\bar{\vartheta})=
 \left \{ \rge(I_3,0),\  \rge(0,I_3)\right\}}\\
\label{eq-33} &\qquad\cup
 \left\{\rge  \left [
 \left(\begin{matrix}
(1-\alpha)I_2+\alpha ww^{T}& 0\\
0 & 0
\end{matrix}\right ),
\left (\begin{matrix}
\alpha(I_2-ww^{T}) &
 -\mathcal{F} w\\
0 & 1
\end{matrix}\right)
\right ]\bigg|
\ \alpha\in[0,1], w \in \mathbb{S}_{\R^2} \right \}.
\end{align}

\end{proposition}
\begin{proof}
The case where $(\bar{x},\bar{g},\bar{\vartheta})\in M_{2}$ is a simple consequence of the relations (\ref{eq-12}) and (\ref{eq-14}). Formula (\ref{eq-32}) follows from (\ref{eq-13}) and \eqref{EqM1Alt1}. Finally, in case of $M_{4}$  one has to  analyze various sequences $(v,g,\vartheta)\to (0,0,0)$ belonging to $L\cup M_1 \cup M_3^+$. For the cases where $(v,g,\vartheta)$ belongs to $L\cup M_3^+$ we can simply apply \eqref{eq-12} and \eqref{eq-13} to calculate the limits. When $(v,g,\vartheta)\in M_1$ one has $-\F\vartheta\,/\,(\norm{v_{12}}-\F\vartheta)\to \alpha\in[0,1]$ and $v_{12}\,/\,\norm{v_{12}}\to w$ for suitable subsequences, and \eqref{eq-33} follows from \eqref{EqM1Alt1}.  All subspaces arising in the above formulas have dimension three, and so the statement has been established.
\end{proof}
\bigskip
Note that for $(\bar{x},\bar{g},\bar{\vartheta})\in M_{4}$
the collection $\Sp\widetilde{Q}(\bar{v},\bar{g},\bar{\vartheta})$ contains an infinite number of subspaces parameterized by $\alpha$ and $w$.
\section*{Declarations}

\noindent{\bf Funding. }The authors express their gratitude for the support from the Austrian Science Fund
(FWF)  grant P29190-N32 (H. Gfrerer, M. Mandlmayr) and the Czech Science Foundation (GACR) grant GA22-15524S (J.V.Outrata), grant GF21-06569K and the  MSMT CR grant 8J21AT001 (J.Valdman).

\noindent{\bf Conflict of interest.} The authors have no competing interests to declare that are relevant to the content of this article.

\noindent{\bf Data availability.}
Data sharing is not applicable to this article as no datasets were generated or analyzed during the current study.

\end{document}